\documentclass[12pt,a4paper,intlimits,leqno]{amsart}
\usepackage{graphicx}
\usepackage{mathrsfs}
\usepackage{verbatim}
\pdfoutput=1
\usepackage[a4paper,pdftex,total={16.2cm,21.6cm}]{geometry}

\numberwithin{equation}{section}
\newtheoremstyle{italic}{\bigskipamount}{\bigskipamount}%
{\itshape}{}{\bfseries}{.}{ }{}
\theoremstyle{italic}
\newtheorem*{definition*}{Definition}
\newtheorem*{lemma*}{Lemma}
\newtheorem{theorem}{Theorem}[section]
\newtheorem{lemma}{Lemma}[section]

\newtheorem{corollary}{Corollary}[section]
\newtheoremstyle{roman}{\bigskipamount}{\bigskipamount}%
{\upshape}{}{\bfseries}{.}{ }{}
\theoremstyle{roman}
\newtheorem{remark}{Remark}[section]
\newtheorem{example}{Example}

\newcommand{\dx}{\, \mathrm{d} x}
\newcommand{\ds}{\, \mathrm{d} s}
\newcommand{\dt}{\, \mathrm{d} t}
\newcommand{\dtau}{\, \mathrm{d} \tau}
\newcommand{\dGamma}{\, \mathrm{d} \Gamma}
\newcommand{\R}{\mathbb{R}}

\newcommand{\N}{\mathbb{N}}

\newcommand{\E}{\mathscr{E}}
\newcommand*{\const}{\operatorname*{const}}
\newcommand*{\hull}{\operatorname*{span}}
\newcommand*{\divergence}{\operatorname*{div}}

\newcommand*{\curl}{\operatorname*{curl}}
\newcommand*{\supp}{\operatorname*{supp}}
\newcommand*{\esssup}{\operatorname*{ess\,sup}}

\begin{document}
\title{Well-posedness of the Maxwell equations \linebreak
  with nonlinear Ohm law}
\author[Jens A.~Griepentrog and Joachim Naumann]%
{Jens A.~Griepentrog and Joachim Naumann*}
\dedicatory{Institut f\"{u}r Mathematik, Humboldt-Universit\"{a}t zu Berlin\\
  Unter den Linden 6, D--10099 Berlin}
\thanks{\emph{Email addresses}: \texttt{griepent@math.hu-berlin.de};
  \texttt{jnaumann@math.hu-berlin.de}}
\thanks{*Corresponding author}
\date{\today}
\begin{abstract}
  This paper is concerned with weak solutions~$(e,h) \in L^2 \times L^2$
  of the Maxwell equations with nonlinear Ohm law
  and under perfect conductor boundary conditions.
  These solutions are defined in terms of integral identities
  with appropriate test functions.
  The main result of our paper is an energy equality
  that holds for \emph{any} weak solution~$(e,h)$.
  The proof of this result makes essential use of the existence
  of time-continuous representatives in the equivalence classes~$(e,h)$.
  As a consequence of the energy equality,
  we prove the well-posedness of the $L^2$-setting of the Maxwell equations
  with regard to the initial-boundary conditions under consideration.
  In addition, we establish the existence of a weak solution
  via the Faedo-Galerkin method.
  An appendix is devoted to the proof of a Carath\'{e}odory solution
  to an initial-value problem for an ordinary differential equation.
\end{abstract}
\maketitle
\thispagestyle{empty}
\section{Introduction}
\noindent
Let $\Omega \subset \R^3$ be a bounded domain with smooth boundary
$\Gamma := \partial\Omega$, and let $0 < T < +\infty$.
The evolution of an electromagnetic field in the cylinder
$Q_T = \Omega \times {]0,T[}$ is governed by the Maxwell equations
\begin{align}
  \partial_t (\varepsilon e)
  \, & = \, \curl h - j,
  \\[\smallskipamount]
  \partial_t (\mu h)
  \, & = \, - \curl e,
\end{align}
where $e = e(x,t)$ and $h = h(x,t)$ ($(x,t) \in Q_T$)
represent the electric and magnetic field, respectively.
The $3 \times 3$ matrices $\varepsilon = \varepsilon(x)$
and $\mu = \mu(x)$ ($x \in \Omega$)
characterize the electric permittivity and the magnetic permeability,
respectively, of the medium under consideration.
The vector field $j$ denotes a current density
(for details, see, e.g.~[12, Ch.~1], [15, Ch.~6], [28, Teil~I, \S\S~3--4]).

In the present paper we consider vector fields $j$ of the form
\begin{equation*}
  j \, = \, j_0(x,t) + j_1(x,t,e),
  \quad (x,t) \in Q_T,
  \quad e \in \R^3.
\end{equation*}
Here, $j_0 = j_0(x,t)$ represents a given current density field,
while $j_1$ characterizes the current density caused by the electric field $e$.
The most common constitutive relations between $j_1$ and $e$ are Ohm laws.
\begin{example}
  The well-known \emph{linear Ohm law} reads
  \begin{equation*}
    j_1 \, = \, \sigma(x,t) \, e,
  \end{equation*}
  where $\sigma = \sigma(x,t)$ denotes a symmetric non-negative $3 \times 3$
  matrix which describes the conductivity of the medium. If
  \begin{equation*}
    \sigma(x,t) \, = \, \sigma_0(x,t) \, \delta
    \quad \big( 0 < \sigma_0(x,t) \le \mathrm{const},
    \quad \text{$\delta = (\delta_{kl})_{k,l = 1,2,3}$ unit matrix} \big),
  \end{equation*}
  it follows $U = IR$, where $U = |e|$ voltage, $I = |j_1|$ current
  and $R = 1/\sigma_0(x,t)$ resistance (see~[28, pp.~19--20]).
  \hfill $\Box$
\end{example}
\begin{example}
  Let $\sigma_0(x,t)$ and $\delta$ be as above. Define
  \begin{equation*}
    \sigma(x,t,|e|)
    \, = \, \frac{\sigma_0(x,t)}{(1 + |e|^2)^{1/2}} \, \delta,
    \quad (x,t,e) \in Q_T \times \R^3.
  \end{equation*}
  Then the \emph{nonlinear Ohm law}
  \begin{equation*}
    j_1 \, = \, \sigma(x,t,|e|) \, e
    \, = \, \frac{\sigma_0(x,t)}{(1 + |e|^2)^{1/2}} \, e
  \end{equation*}  
  models the effect of ``asymptotic saturation of current at large voltages''
  in certain semi-conductors, i.e.
  \begin{equation*}
    I = |j_1| \nearrow \sigma_0(x,t)
    \quad\text{for $U = |e|$ increasing}.
  \end{equation*}
  We note that the mapping
  \begin{equation*}
    e \longmapsto \frac{1}{(1 + |e|^2)^{1/2}} \, e,
    \quad e \in \R^3,
  \end{equation*}
  is strictly monotone, i.e.~for all $e, \bar{e} \in \R^3$, $e \neq \bar{e}$,
  \begin{equation*}
    \left( \frac{e}{(1 + |e|^2)^{1/2}}
      - \frac{\bar{e}}{(1 + |\bar{e}|^2)^{1/2}}
    \right) \cdot (e - \bar{e})
    \, = \, \frac{1}{(1 + |e|^2)^{1/2}} \, |e - \bar{e}|^2
    \, > \, 0
    \quad\text{if $|e| = |\bar{e}|$}
  \end{equation*}
  and
  \begin{multline*}
    \quad\, \left( \frac{e}{(1 + |e|^2)^{1/2}}
      - \frac{\bar{e}}{(1 + |\bar{e}|^2)^{1/2}}
    \right) \cdot (e - \bar{e})
    \\[\medskipamount]
    \ge \, \left( \frac{|e|}{(1 + |e|^2)^{1/2}}
      - \frac{|\bar{e}|}{(1 + |\bar{e}|^2)^{1/2}}
    \right) \cdot (|e| - |\bar{e}|)
    \, > \, 0
    \quad\text{if $|e| \neq |\bar{e}|$}. \quad
  \end{multline*}
  \hfill $\Box$
\end{example}
\noindent
Let $\sigma: Q_T \times \R^+ \longrightarrow \R^{3 \times 3}$
satisfy the following two conditions
\begin{itemize}
\item[(a)] \emph{growth}:
  for all $(x,t,e) \in Q_T \times \R^3$,
  \begin{equation*}
    \big| \sigma(x,t,|e|) \, e \big| \, \le \, c_1 |e|,
    \quad c_1 = \const > 0;
  \end{equation*}
\item[(b)] \emph{monotonicity}:
  for all $(x,t) \in Q_T$ and all $e, \bar{e} \in \R^3$,
  \begin{equation*}
    \big( \sigma(x,t,|e|) \, e - \sigma(x,t,|\bar{e}|) \, \bar{e} \big)
    \cdot (e - \bar{e}) \, \ge \, 0.
  \end{equation*}
\end{itemize}
Then the Ohm law
\begin{equation*}
  j_1 \, = \, \sigma(x,t,|e|) \, e
\end{equation*}
includes Examples~1 and~2 as special cases.
For developing our $L^2$-theory of~(1.1)--(1.4),
below we further generalize conditions~(a) and~(b)
(see hypotheses~(H1)--(H3) in Section~2, and hypothesis~(H7) in Section~4).

Remarks on nonlinear Ohm laws can be also found in~[12, p.~14]
and~[31, pp.~256--257] (see also the references listed in this paper).
In~[16], the author studies~(1.1), (1.2)
with $e \longmapsto j_1(\cdot,e)$ monotone and of class $C^1$.
\hfill $\Box$

\bigskip\noindent
Let $n = n(x)$ denote the outward directed unit normal at $x \in \Gamma$.
We complement system~(1.1), (1.2) by the boundary and initial conditions
\begin{gather}
  n \times e \, = \, 0
  \quad\text{on $\Gamma \times {]0,T[}$},
  \\[\medskipamount]
  e \, = \, e_0,
  \quad h \, = \, h_0
  \quad\text{in $\Omega \times \{0\}$},
\end{gather}
where $(e_0, h_0)$ are given data.
Boundary condition~(1.3) models a perfect conductor.
A brief discussion of boundary conditions for the Maxwell equations
can be found in~[28, p.~30].
The author points out that both boundary condition~(1.3)
and the boundary condition $n \times h = 0$ on $\Gamma \times {]0,T[}$ imply
vanishing of the integral $\int_{\Gamma} n \cdot S \dGamma$ (see Section~2).
\hfill $\Box$

\bigskip\noindent
For notational simplicity, in what follows we write $j(x,t,e)$
(or briefly $j(e)$) in place of $j(x,t,e(x,t))$ ($(x,t) \in Q_T$).

We multiply scalarly~(1.1) and~(1.2) by $e$ and $h$, respectively,
and add the equations so obtained. Thus
\begin{equation}
  \frac{1}{2} \, \frac{\partial}{\partial t}
  \big( (\varepsilon e) \cdot e + (\mu h) \cdot h \big)
  + \divergence S + j(e) \cdot e 
  \, = \, 0 \quad\text{in $Q_T$},
\end{equation}
where
\begin{equation*}
  S \, := \, e \times h \; \footnotemark
\end{equation*}
\footnotetext{Note that
  $a \cdot (b \times c) = b \cdot (c \times a) = c \cdot (a \times b)$
  for any $a, b, c \in \R^3$,
  and $\divergence(u \times v) = v \cdot \curl u - u \cdot \curl v$
  for any $u, v \in C^1$.}%
denotes the \emph{Poynting vector of $(e,h)$}. The field $S$
represents the flux of electromagnetic energy through $Q_T$.

Integration of~(1.5) over $\Omega \times [0,t]$ ($0 \le t \le T$) gives
\begin{equation}
  \E(t)
  + \int_0^t \int_\Omega \divergence S \dx \ds
  + \int_0^t \int_\Omega j(e) \cdot e \dx \ds
  \, = \, \E(0),
  \quad t \in [0,T],
\end{equation}
where
\begin{align*}
  \E(t)
  & \, := \, \frac{1}{2} \int_\Omega
  \big( (\varepsilon e)(x,t) \cdot e(x,t) + (\mu h)(x,t) \cdot h(x,t)
  \big) \dx,
  \\[\medskipamount]
  \E(0)
  & \, = \, \frac{1}{2} \int_\Omega
  \big( (\varepsilon e_0)(x) \cdot e_0(x) + (\mu h_0)(x) \cdot h_0(x)
  \big) \dx
\end{align*}
(cf.~(1.4)). The non-negative function $\E(t)$ represents
the electromagnetic energy of $(e,h)$ at time $t$.
Equation~(1.6) is called \emph{balance of electromagnetic energy}
(or \emph{Poynting theorem}). The term $j(e) \cdot e$ in equation~(1.6)
characterizes the conversion of electromagnetic energy into heat
(see, e.g. [15, pp.~236--237], [28, pp.~25--26]). 

We next combine the divergence theorem with boundary condition~(1.3) to obtain
\begin{equation*}
  \int_\Omega (\divergence S)(x,t) \dx
  \, = \, \int_\Gamma n(x) \cdot S(x,t) \dGamma
  \, = \, 0
\end{equation*}
for all $t \in {]0,T[}$.
Thus, equation~(1.6) turns into the \emph{energy equality}
\begin{equation}
  \E(t)
  + \int_0^t \int_\Omega j(e) \cdot e \dx \ds
  \, = \, \E(0),
  \quad t \in [0,T]
\end{equation}
(see, e.g.~[12], [28]). 

The equations~(1.6) and~(1.7) are fundamental to the theory
of electromagnetism. This aspect has been discussed in great detail
(with $\R^3$ in place of $\Omega$) by \textsc{J.~C.~Maxwell}
in his celebrated work~[23, pp.~486--488].
\hfill $\Box$

We note that the scalar function $\E(t)$ introduced above
is well-defined (for a.e.~$t \in [0,T]$) for vector fields
$(e,h) \in L^2(Q_T)^3 \times L^2(Q_T)^3$,
provided the entries of the matrices $\varepsilon(\cdot)$ and $\mu(\cdot)$
are bounded measurable functions in $\Omega$.
\section{Weak solutions of (1.1)--(1.4)}
\subsection*{Integral identities for classical solutions}
Let $\Omega \subset \R^3$ be a bounded domain with smooth boundary $\Gamma$.
To motivate the definition of weak solutions of~(1.1)--(1.4)
which will be introduced below, we consider a classical solution
$(e,h) \in C^1(\overline{Q}_T)^3 \times C^1(\overline{Q}_T)^3$
of~(1.1)--(1.4) and test functions
$(\Phi,\Psi) \in C^1(\overline{Q}_T)^3 \times C^1(\overline{Q}_T)^3$ such that
\begin{equation}
  \Phi(\cdot,T) \, = \, \Psi(\cdot,T) \, = \, 0
  \quad\text{in $\Omega$}.
\end{equation}
We multiply~(1.1) and~(1.2) scalarly by $\Phi$ and $\Psi$,
respectively, integrate over $Q_T$ and integrate by parts
with respect to $t$ the terms $\partial_t (\varepsilon e) \cdot \Phi$
and $\partial_t (\mu h) \cdot \Psi$. 
Observing~(2.1) and initial conditions~(1.4) we obtain
\begin{gather}
  -\int_{Q_T}
  (\varepsilon e) \cdot \partial_t \Phi \dx \dt
  + \int_{Q_T} \big( {-}\curl h + j(e) \big) \cdot \Phi \dx \dt
  \, = \, \int_\Omega
  (\varepsilon e_0)(x) \cdot \Phi(x,0) \dx,
  \\[\medskipamount]
  -\int_{Q_T}
  (\mu h) \cdot \partial_t \Psi \dx \dt
  + \int_{Q_T} (\curl e) \cdot \Psi \dx \dt
  \, = \, \int_\Omega
  (\mu h_0)(x) \cdot \Psi(x,0) \dx.
\end{gather}

Next, we apply the Green formula
\begin{equation}
  \int_\Omega (\curl a) \cdot b \dx
  - \int_\Omega a \cdot \curl b \dx
  \, = \, \int_\Gamma (n \times a) \cdot b \dGamma,
  \quad a, b \in C^1(\overline{\Omega})^3
\end{equation}
to
\begin{equation*}
  a = -h(\cdot,t),
  \quad b = \Phi(\cdot,t)
  \quad\text{such that}\quad n \times \Phi(\cdot,t) = 0
  \quad\text{on $\Gamma \times {]0,T[}$}
\end{equation*}
resp.
\begin{equation*}
  a = e(\cdot,t) \quad\text{(observing~(1.3))},
  \quad b = \Psi(\cdot,t)
\end{equation*}
($t \in {]0,T[}$) in the second integral of the left-hand side
of~(2.2) and~(2.3).
Thus, (2.2) and~(2.3) turn into the \emph{integral identities}
\begin{gather}
  -\int_{Q_T}
  (\varepsilon e) \cdot \partial_t \Phi \dx \dt
  + \int_{Q_T} \big( {-}h \cdot \curl \Phi
  + j(e) \cdot \Phi \big) \dx \dt
  \, = \, \int_\Omega
  (\varepsilon e_0)(x) \cdot \Phi(x,0) \dx,
  \\[\medskipamount]
  -\int_{Q_T}
  (\mu h) \cdot \partial_t \Psi \dx \dt
  + \int_{Q_T} e \cdot \curl \Psi \dx \dt
  \, = \, \int_\Omega
  (\mu h_0)(x) \cdot \Psi(x,0) \dx.
\end{gather}

If the entries of the matrices $\varepsilon(\cdot)$ and $\mu(\cdot)$
are bounded measurable functions in $\Omega$, if $j(e) \in L^2(Q_T)^3$
and $(e_0,h_0) \in L^2(\Omega)^3 \times L^2(\Omega)^3$,
then all the integrals in~(2.5) and~(2.6) are well-defined
for $(e,h) \in L^2(Q_T)^3 \times L^2(Q_T)^3$ and an appropriate class
of test functions $(\Phi,\Psi)$. More specifically,
let $\Phi \in C^1(\overline{Q}_T)^3$ satisfy~(2.1) and suppose that
\begin{equation}
  \int_\Omega (\curl \Phi(\cdot,t)) \cdot z \dx
  \, = \, \int_\Omega \Phi(\cdot,t) \cdot \curl z \dx
  \quad\text{for all $t \in {]0,T[}$ and all $z \in C^1(\overline{\Omega})^3$}.
\end{equation}
Clearly, (2.7) holds true when $n \times \Phi = 0$
on $\Gamma \times {]0,T[}$ (see~(2.4)).
We note that~(2.7) does make sense regardless of
whether the boundary $\Gamma$ is smooth or not.

Thus, appropriate conditions for $\Phi$ and $\Psi$ are
\begin{equation*}
  \text{$\curl \Phi \in L^2(Q_T)^3$ satisfies~(2.7)},
  \quad \curl \Psi \in L^2(Q_T)^3.
\end{equation*}
\subsection*{Definition of weak solutions}
Let $\Omega \subset \R^3$ be an open set. We define
\begin{align*}
  V & := \Bigg\{ u \in L^2(\Omega)^3; \,
      \text{there exists $F \in L^2(\Omega)^3$ such that}
  \\
    & \qquad \int_\Omega u \cdot \curl \varphi \dx
      \, = \, \int_\Omega F \cdot \varphi \dx
      \;\;\text{for all $\varphi \in C^\infty_c(\Omega)^3$}
      \Bigg\},
\end{align*}
i.e., the vector field $u \in L^2(\Omega)^3$ is in $V$,
if the distribution $\curl u$ can be represented by $F \in L^2(\Omega)^3$.
We identify this distribution with $F$.
The space $V$ is usually denoted by $H(\curl;\Omega)$.
It is a  Hilbert space with respect to the scalar product
\begin{equation*}
  (u,v)_V \, := \,
  \int_\Omega \big( u \cdot v + (\curl u) \cdot \curl v \big) \dx.
\end{equation*}
We next define the closed subspace
\begin{equation*}
  V_0 \, := \, \Bigg\{ u \in V;
  \int_\Omega (\curl u) \cdot \psi \dx
  \, = \, \int_\Omega u \cdot \curl \psi \dx
  \;\;\text{for all $\psi \in V$} \Bigg\}.
\end{equation*}
To our knowledge, the analogue of this space with $H^1(\Omega)^3$ in place
of~$V$ has been introduced for the first time in~[20, pp.~215--216]
and was then used by other authors, see e.g.~[16] and~[30].
\begin{remark}
  1.~For $u \in L^2(\Omega)^3$ the following conditions are equivalent:
  \begin{itemize}
  \item[(i)] $u \in V_0$;
    \smallskip
  \item[(ii)] \emph{there exists $G = G(u) \in L^2(\Omega)^3$ such that}
    \begin{equation}
      \int_\Omega u \cdot \curl \psi \dx
      \, = \, \int_\Omega G \cdot \psi \dx
      \quad\text{\emph{for all $\psi \in V$}}.
    \end{equation}
  \end{itemize}
  To prove this it suffices to show (ii) $\Longrightarrow$ (i).
  The equation in~(2.8) evidently holds for all $\psi \in C^\infty_c(\Omega)^3$.
  This means that the distribution $\curl u$ is represented by $G$.
  Hence, $u \in V$. Again appealing to~(2.8) gives
  \begin{equation*}
    \int_\Omega (\curl u) \cdot \psi \dx
    \, = \, \int_\Omega G \cdot \psi \dx
    \, = \, \int_\Omega u \cdot \curl \psi \dx
    \quad\text{for all $\psi \in V$}, 
  \end{equation*}
  i.e.~$u \in V_0$.
  \hfill $\Box$

  2.~Define
  \begin{equation*}
    W_0 \, := \,
    \text{closure of $C^\infty_c(\Omega)^3$ in $V$}.
  \end{equation*}
  It is readily seen that $W_0 \subset V_0$. In fact, we have
  \begin{equation*}
    W_0 \, = \, V_0.
  \end{equation*}
  Following an argument by~[8, Ch.~IX, \S~1.2, Proof of~Thm.~2, p.~207],
  take $u_0 \in V_0$ such that $(u_0,w)_V = 0$ for all $w \in W_0$.
  Writing $\psi_0 = \curl u_0$ it follows
  \begin{equation*}
    \int_\Omega \psi_0 \cdot \curl \varphi \dx
    \, = \, - \int_\Omega u_0 \cdot \varphi \dx
    \quad\text{for all $\varphi \in C^\infty_c(\Omega)^3$}.
  \end{equation*}
  Thus, $\psi_0 \in V$ and $\curl \psi_0 = -u_0$.
  Therefore, by the definition of $V_0$,
  \begin{equation*}
    \int_\Omega |u_0|^2 \dx
    \, = \, - \int_\Omega u_0 \cdot \curl \psi_0 \dx
    \, = \, - \int_\Omega (\curl u_0) \cdot \psi_0 \dx
    \, = \, - \int_\Omega |\!\curl u_0|^2 \dx.
  \end{equation*}
  Whence, $u_0 = 0$.
  \hfill $\Box$
\end{remark}
If $\Omega \subset \R^3$ is an open set the boundary of which
is locally representable by Lipschitz graphs,
then the space $V_0$ is usually denoted by $H_0(\curl;\Omega)$
(cf.~[14, Thm.~2.11, Thm.~2.12, pp.~34--35], [8, pp.~204--206]).
\begin{remark}
  An example of a bounded domain the boundary of which
  \emph{cannot} be represented locally by Lipschitz graphs
  can be found in~[24, p.~39, Fig.~3.1 (``crossed bricks'')].
  Domains of this type seem to be relevant in electrical engineering.
  We note that our approach to the weak formulation of~(1.1)--(1.4)
  based on the spaces~$V$ and~$V_0$ we introduced above,
  does not make any assumption on the boundary of the underlying domain.
  In particular, this approach suits well to an energy equality of type~(1.7).
  \hfill $\Box$
\end{remark}
\begin{remark}
  Let the boundary $\Gamma$ be locally representable by Lipschitz graphs.
  Then there exists a linear continuous mapping
  $\gamma_\tau: V \longrightarrow H^{-1/2}(\Gamma)^3$
  \footnote{For the definition and the properties of the spaces $H^s(\Gamma)$
    ($s > 0$ real) see, e.g., [26, Ch.~2, \S\S~3.8, 5.4].
    By $\langle z^*,z \rangle_{H^{1/2}(\Gamma)^3}$
    we denote the value of $z^* \in H^{-1/2}(\Gamma)^3$
    (dual space of $H^{1/2}(\Gamma)^3)$ at $z \in H^{1/2}(\Gamma)^3$.}%
  such that
  \begin{gather*}
    \gamma_\tau(u) \, = \, n \times (u|_\Gamma)
    \quad\text{for all $u \in C^1(\overline{\Omega})^3$},
    \\[\medskipamount]
    \int_\Omega (\curl u) \cdot \psi \dx
    - \int_\Omega u \cdot \curl \psi \dx
    \, = \, \langle \gamma_\tau(u), \psi \rangle_{H^{1/2}(\Gamma)^3}
    \quad\text{for all $u \in V$ and all $\psi \in H^1(\Omega)^3$}
  \end{gather*}
  (see, e.g.~[1], [8, Ch.~IX, \S~1.2], [24, Thm.~3.26, Thm.~3.33]). It follows
  \begin{equation*}
    V_0 \, = \, \big\{ u \in V; \;
    \gamma_\tau(u) = 0 \;\; \text{in $H^{-1/2}(\Gamma)^3$} \big\}.
  \end{equation*}
  For a precise description of the image
  of the tangential trace mapping~$\gamma_\tau$, cf.~[4], [5].
  \hfill $\Box$
\end{remark}
\noindent
We introduce more notations.
Let $X$ be a real normed space with norm $|\cdot|_X$.
By $L^p(0,T;X)$ $(1 \le p \le +\infty)$ we denote the vector space
of equivalence classes of strongly measurable functions
$u: [0,T] \longrightarrow X$ such that the function $t \longmapsto |u(t)|_X$
is in $L^p(0,T)$. The norm on $L^p(0,T;X)$ is given by
\begin{equation*}
  \|u\|_{L^p(0,T;X)} \, := \,
  \begin{cases}
    \, \left( \displaystyle\int_0^T |u(t)|_X^p \dt \right)^{1/p}
    & \text{if $1 \le p < +\infty$},
    \\[\medskipamount]
    \, \displaystyle\esssup_{t \in ]0,T[} |u(t)|_X
    & \text{if $p = +\infty$},
  \end{cases}
\end{equation*}
(for details see, e.g.~[2], [3, Appendice, pp.~137--140], [9], [32]).
If $X$ is a Banach space, then $L^p(0,T;X)$ does.

Let $H$ be a real Hilbert space with scalar product $(\cdot\,,\cdot)_H$.
Then $L^2(0,T;H)$ is a Hilbert space with respect to the scalar product
\begin{equation*}
  (u,v)_{L^2(0,T;H)}
  \, := \, \int_0^T (u(t),v(t))_H \dt.
\end{equation*}

Given $u \in L^p(Q_T)$ $(1 \le p < +\infty)$, we define
\begin{equation*}
  [u](t) \, := \, u(\cdot,t)
  \quad\text{for a.e. $t \in [0,T]$}.
\end{equation*}
By the Fubini theorem, $[u] \in L^p(0,T;L^p(\Omega))$ and
\begin{equation*}
  \int_{Q_T} |u(x,t)|^p \dx \dt
  \, = \, \int_0^T \|[u](t)\|_{L^p(\Omega)}^p \dt.
\end{equation*}
It is easily seen that the map $u \mapsto [u]$ is a linear isometry
from $L^p(Q_T)$ onto $L^p(0,T;L^p(\Omega))$.
Throughout our paper we identify these spaces.
\hfill $\Box$

\bigskip\noindent
To introduce the notion of weak solutions of (1.1)--(1.4),
we make the following hypotheses on $\varepsilon$, $\mu$
in~(1.1), (1.2), the field~$j$, and $(e_0,h_0)$ in~(1.4):
\begin{align}
  & \left\{
    \begin{aligned}
      \, & \text{\emph{the entries of the $3 \times 3$ matrices
          $\varepsilon(\cdot)$ and $\mu(\cdot)$}}
      \\
      & \text{\emph{are bounded measurable functions in $\Omega$}};
    \end{aligned}
        \tag{H1}
        \right.
  \\[\medskipamount]
  & \quad j(x,t,\xi) \, = \, j_0(x,t) + j_1(x,t,\xi),
    \quad (x,t,\xi) \in Q_T \times \R^3,
    \notag
    \intertext{\emph{where}}
  & \left\{
           \begin{aligned}
      \, & \text{$j_0 \in L^2(Q_T)^3$},
      \\
      & \text{$j_1: Q_T \times \R^3 \longrightarrow \R^3$
        \emph{is a Carath\'{e}odory function, i.e.}},
      \\
      & \text{$(x,t) \longmapsto j_1(x,t,\xi)$
        \emph{is measurable in $Q_T$ for all $\xi \in \R^3$}},
      \\
      & \text{$\xi \longmapsto j_1(x,t,\xi)$
        \emph{is continuous in $\R^3$ for a.e.~$(x,t) \in Q_T$}};
    \end{aligned}
        \tag{H2}
        \right.
  \\[\medskipamount]
  & \left\{
    \begin{aligned}
      \, & \text{\emph{there exists $c_1 = \const > 0$ such that}}
      \\
      & \text{$|j_1(x,t,\xi)| \le c_1 |\xi|$
        \quad \emph{for all $(x,t,\xi) \in Q_T \times \R^3$}};
    \end{aligned}
        \tag{H3}
        \right.
\end{align}
\emph{and}
\begin{equation}
  (e_0, h_0) \in L^2(\Omega)^3 \times L^2(\Omega)^3.
  \tag{H4}
\end{equation}
\begin{remark}
  1.~Given any measurable vector field $u:Q_T \longrightarrow \R^3$,
  from~(H2) it follows that the mapping
  $(x,t) \longmapsto j(x,t,u(x,t))$ is measurable in $Q_T$.
  Hence, by~(H3),
  \begin{equation*}
    j(u) \, = \, j(\cdot\,, \cdot\,, u(\cdot\,,\cdot)) \in L^2(Q_T)^3
    \quad\text{for all $u \in L^2(Q_T)^3$}.
  \end{equation*}

  2.~Hypotheses~(H2), (H3) on $j_1$ include the Ohm laws
  considered in Examples~1 and~2 in Section~1.
  \hfill $\Box$
\end{remark}

\noindent
The following definition extends integral identities~(2.5)
and~(2.6) to the $L^2$-framework.
\begin{definition*}
  Let hypotheses~\emph{(H1)--(H4)} hold. The pair
  \begin{equation*}
    (e,h) \in L^2(Q_T)^3 \times L^2(Q_T)^3
  \end{equation*}
  is called \emph{weak solution of~(1.1)--(1.4)} if
  \begin{align}
    & \left\{
      \begin{aligned}
        & - \! \displaystyle \int_{Q_T}
        (\varepsilon e) \cdot \partial_t \Phi \dx \dt
        + \int_{Q_T} \! \big( {-}h \cdot \curl \Phi
        + j(e) \cdot \Phi \big) \dx \dt
        \, = \, \int_\Omega
        (\varepsilon e_0)(x) \cdot \Phi(x,0) \dx
        \\
        & \quad\text{for all $\Phi \in L^2(0,T;V_0)$
          such that $\partial_t \Phi \in L^2(Q_T)^3$
          and $\Phi(\cdot,T) = 0$ a.e.~in $\Omega$};
      \end{aligned}
          \right.
    \\[\medskipamount]
    & \left\{
      \begin{aligned}
        & - \! \displaystyle \int_{Q_T}
        (\mu h) \cdot \partial_t \Psi \dx \dt
        + \int_{Q_T} e \cdot \curl \Psi \dx \dt
        \, = \, \int_\Omega
        (\mu h_0)(x) \cdot \Psi(x,0) \dx
        \\
        & \quad\text{for all $\Psi \in L^2(0,T;V)$
          such that $\partial_t \Psi \in L^2(Q_T)^3$
        and $\Psi(\cdot,T) = 0$ a.e.~in $\Omega$}.
    \end{aligned}
          \right.
  \end{align}
\end{definition*}
Let $\Gamma = \partial \Omega$ be smooth.
Then from the discussion above it follows that every classical solution
of~(1.1)--(1.4) is a weak solution of this problem, too, cf.~(2.5), (2.6).
We note that our definition of weak solutions basically coincides with
the definitions introduced in~[10, Ch.~VII, \S~4.2], [11], [12, p.~326], [17].

In case of \emph{linear Ohm laws}, existence theorems for weak solutions
of~(1.1)--(1.4) are established in~[10, Ch.~VII, \S~4.3]
(cf.~also Section~5 below), [11] and~[12, Ch.~7, \S~8.3].
In~[29], the author proves the local well-posedness
of~(1.1)--(1.4) for a class of nonlinear Maxwell equations
in spaces of differentiable functions.

The aim of the present paper is to prove that for any initial datum
$(e_0,h_0) \in L^2(\Omega)^3 \times L^2(\Omega)^3$
(with $\Omega$ possibly unbounded),
\emph{every} weak solution $(e,h) \in L^2(Q_T)^3 \times L^2(Q_T)^3$
of~(1.1)--(1.4) in the sense of the above definition
\begin{itemize}
\item has a representative in $C\big([0,T];L^2(\Omega)^3\big)
  \times C\big([0,T];L^2(\Omega)^3\big)$,
  \smallskip
\item obeys an energy equality (which implies well-posedness) and
  \smallskip
\item can be obtained as limit of Faedo-Galerkin approximations.
\end{itemize}
\subsection*{Existence of the distributional derivatives
  $\boldsymbol{(\varepsilon e)'}$ and $\boldsymbol{(\mu h)'}$}
We will prove that~(2.9) and~(2.10) imply the existence
of the $t$-derivatives of $\varepsilon e$ and $\mu h$
in the sense of vector-valued distributions.
To this end, we introduce some more notation.

Let $X$ be a real normed space.
By $X^*$ we denote the dual space of $X$, and by $\langle x^*, x \rangle_X$
the dual pairing between $x^* \in X^*$ and $x \in X$.
Let $H$ be a real Hilbert space with scalar product $(\cdot \,, \cdot)_H$
and suppose that $X$ is continuously and densely embedded into $H$.
We identify $H$ with its dual space $H^*$ via the Riesz representation theorem
to obtain
\begin{gather}
  H \subset X^*
  \quad\text{continuously},
  \notag \\[\smallskipamount]
  \langle z, x \rangle_X
  \, = \, (z, x)_H
  \quad\text{for all $z \in H$ and all $x \in X$}
\end{gather}
(cf.~[32, Ch.~23, \S~4]).
If $X$ is reflexive, then $H \subset X^*$ densely.

Next, let $X$ and $Y$ be two real normed spaces such that $X \subset Y$
continuously and densely. Given $u \in L^1(0,T;X)$,
we identify $u$ with an element in $L^1(0,T;Y)$ and denote it by $u$ again.
An element $U \in L^1(0,T;Y)$ will be called
\emph{derivative of $u$ in the sense of distributions from $[0,T]$ into $Y$} if
\begin{equation*}
  \int_0^T \dot{\zeta}(t) \, u(t) \dt
  \, = \, - \int_0^T \zeta(t) \, U(t) \dt
  \quad\text{in $Y$}
\end{equation*}
for all $\zeta \in C_c^\infty({]0,T[})$ and denoted by
\begin{equation*}
  u' \, := \, U
\end{equation*}  
(see~[3, Appendice, Prop.~A.6, p.~154], [9, Ch.~2.1],
[21, Ch.~1, \S~1.3] and~[32, Ch.~23, \S\S~5--6]).
The derivative $u'$ is uniquely determined, if $Y^*$ is separable.
If $Y$ is reflexive, then there exists an absolutely continuous representative
$\tilde{u}:[0,T] \longrightarrow Y$ in the equivalence class $L^1(0,T;Y)$
such that
\begin{equation}
  \tilde{u}(t) \, = \, \tilde{u}(0)
  + \displaystyle \int_0^t w(s) \ds
  \quad\text{for all $t \in [0,T]$}
\end{equation}
(see [3, Appendice, Prop.~A.3, p.~145]).

Let $X$ and $H$ be as above and suppose that $X \subset H$
continuously and densely. 
Then we have the following formula of integration by parts
\begin{equation}
  \left\{
    \begin{aligned}
      & \, \text{for every $u \in L^1(0,T;H)$ such that $u' \in L^1(0,T;X^*)$},
      \\
      \, & \displaystyle \int_0^T
      \langle \alpha(t) \, u'(t), x \rangle_X \dt
      \, = \, \langle \alpha(T) \, \tilde{u}(t)
      - \alpha(0) \, \tilde{u}(0), x \rangle_X
      - \int_0^T (\dot{\alpha}(t) \, u(t), x)_H \dt
      \\
      & \, \text{for all $\alpha \in C^1([0,T])$ and all $x \in X$}.
    \end{aligned}
  \right.
\end{equation}
This formula is easily seen by routine arguments
and observing~(2.11) and~(2.12).
We will need~(2.13) for the proof of Theorem~2.1.
\hfill $\Box$

\bigskip\noindent
We make use of the above notations with
\begin{equation*}
  X = V \quad\text{resp.}\quad X = V_0,
  \quad\text{and}\quad
  H = L^2(\Omega)^3,
\end{equation*}
where $H$ is furnished with the standard scalar product
\begin{equation*}
  (u,v)_H \, := \, \int_\Omega u(x) \cdot v(x) \dx.
\end{equation*}
Then
\begin{equation*}
  H \subset V^*
  \quad\text{resp.}\quad
  H \subset V_0^*
  \quad\text{continuously and densely}.
\end{equation*}
\begin{theorem}
  Let hypotheses \emph{(H1)--(H4)} be satisfied.
  Then for any weak solution
  \begin{equation*}
    (e,h) \in L^2(Q_T)^3 \times L^2(Q_T)^3
  \end{equation*}
  of~\emph{(1.1)--(1.4)} there exist the distributional derivatives
  \begin{equation}
    (\varepsilon e)' \in L^2(0,T;V_0^*),
    \quad (\mu h)' \in L^2(0,T;V^*).
  \end{equation}
  For a.e. $t \in [0,T]$ these derivatives satisfy the identities
  \begin{gather}
    \langle (\varepsilon e)'(t), \varphi \rangle_{V_0}
    + \int_\Omega \big( {-}h(\cdot,t) \cdot \curl \varphi
    + j(e(\cdot,t)) \cdot \varphi \big) \dx
    \, = \, 0
    \quad\text{for all $\varphi \in V_0$},
    \\[\medskipamount]
    \langle (\mu h)'(t), \psi \rangle_V
    + \int_\Omega e(\cdot,t) \cdot \curl \psi \dx
    \, = \, 0
    \quad\text{for all $\psi \in V$}.
  \end{gather}
  The absolutely continuous representatives
  \begin{equation*}
    \widetilde{\varepsilon e}: [0,T] \longrightarrow V_0^*,
    \quad \widetilde{\mu h}: [0,T] \longrightarrow V^*
  \end{equation*}
  in $\varepsilon e$, $\mu h \in L^2(0,T;H)$ fulfill the initial conditions
  \begin{equation}
    (\widetilde{\varepsilon e})(0) = \varepsilon e_0
    \;\; \text{in $V_0^*$},
    \quad (\widetilde{\mu h})(0) = \mu h_0
    \;\; \text{in $V^*$}.
  \end{equation}
  Moreover, for a.e. $t \in [0,T]$,
  \begin{equation}
    \| (\varepsilon e)'(t) \|_{V_0^*}
    \, \le \, \| h(\cdot,t) \|_H
    + \| j(e(\cdot,t)) \|_H,
    \quad \| (\mu h)'(t) \|_{V^*}
    \, \le \, \| e(\cdot,t) \|_H.
  \end{equation}
\end{theorem}
\begin{proof}
  We identify $\varepsilon e \in L^2(0,T;H)$ with an element
  of the space $L^2(0,T;V_0^*)$ and deduce from~(2.9) the existence
  of the distributional derivative $(\varepsilon e)' \in L^2(0,T;V_0^*)$
  and~(2.15) for a.e. $t \in [0,T]$.

  Define $\mathscr{F} = \mathscr{F}(e,h) \in (L^2(0,T;V_0))^*$ by
  \begin{equation*}
    \langle \mathscr{F}, \Phi \rangle_{L^2(0,T;V_0)}
    \, := \, \int_{Q_T}
    \big( {-}h \cdot \curl \Phi + j(e) \cdot \Phi \big) \dx \dt,
    \quad \Phi \in L^2(0,T;V_0).
  \end{equation*}
  The linear isometry $(L^2(0,T;V_0))^* \cong L^2(0,T;V_0^*)$
  enables us to identify $\mathscr{F}$ with its isometric image
  in $L^2(0,T;V_0^*)$ which will be denoted by $\mathscr{F}$ again.
  Thus, $\mathscr{F}(t) \in V_0^*$ for a.e. $t \in [0,T]$ and
  \begin{equation*}
    \langle \mathscr{F}, \Phi \rangle_{L^2(0,T;V_0)}
    \, = \, \int_0^T \langle \mathscr{F}(t), \Phi(t) \rangle_{V_0} \dt
    \quad\text{for all $\Phi \in L^2(0,T;V_0)$}.
  \end{equation*}

  Given any $\varphi \in V_0$ and $\zeta \in C_c^\infty({]0,T[})$,
  we insert $\Phi = \Phi(x,t) = \varphi(x) \, \zeta(t)$
  ($(x,t) \in Q_T$) into~(2.9) to obtain
  \begin{align*}
    \Bigg\langle \int_0^T
    \dot{\zeta}(t) (\varepsilon e)(t) \dt, \varphi
    \Bigg\rangle_{\!\! V_0}
    & \, = \, \int_0^T
      \big( \dot{\zeta}(t) (\varepsilon e)(t), \varphi \big){}_H \dt
      \quad\text{(by~[32, pp.~420--421]; (2.11))}
    \\[\medskipamount]
    & \, = \, \int_{Q_T}
      \big( {-}h \cdot \curl \Phi + j(e) \cdot \Phi \big) \dx \dt
      \quad\text{(by~(2.9))}
    \\
    & \, = \, \langle \mathscr{F}, \Phi \rangle_{L^2(0,T;V_0)}
    \, = \, \Bigg\langle \int_0^T
    \zeta(t) \, \mathscr{F}(t) \dt, \varphi
    \Bigg\rangle_{\!\! V_0}.
  \end{align*}
  Hence,
  \begin{equation*}
    \int_0^T \dot{\zeta}(t) (\varepsilon e)(t) \dt
    \, = \, \int_0^T \zeta(t) \, \mathscr{F}(t) \dt
    \quad\text{in $V_0^*$},
  \end{equation*}
  i.e., $\varepsilon e$ has the distributional derivative
  \begin{equation*}
    (\varepsilon e)' \, = \, -\mathscr{F}
    \, \in \, L^2(0,T;V_0^*). 
  \end{equation*}
  This equation is equivalent to
  \begin{equation}
    \langle (\varepsilon e)'(t), \varphi \rangle_{V_0}
    \, = \, \langle -\mathscr{F}(t), \varphi \rangle_{V_0}
  \end{equation}
  for a.e.~$t \in [0,T]$ and all $\varphi \in V_0$,
  where the set of those $t$ for which~(2.19) fails,
  does not depend on $\varphi$. Whence, (2.15).

  We identify $\mu h \in L^2(0,T;H)$ with an element in $L^2(0,T;V^*)$
  and define $\mathscr{G} = \mathscr{G}(e) \in (L^2(0,T;V))^*$ by
  \begin{equation*}
    \langle \mathscr{G}, \Psi \rangle_{L^2(0,T;V)}
    \, := \, \int_{Q_T} e \cdot \curl \Psi \dx \dt,
    \quad \Psi \in L^2(0,T;V).
  \end{equation*}
  By an analogous reasoning as above
  we obtain the existence of the distributional derivative
  \begin{equation*}
    (\mu h)' \, = \, -\mathscr{G}
    \, \in \, L^2(0,T;V^*).
  \end{equation*}
  This equation is equivalent to~(2.16).

  We identify $\varepsilon e$, $\mu h \in L^2(0,T;H)$
  with elements in $L^2(0,T;V_0^*)$ and $L^2(0,T;V^*)$, respectively.
  Then~(2.14) implies the existence of absolutely continuous
  representatives from $[0,T]$ into $V_0^*$ and $V^*$,
  respectively (cf.~(2.12)).
  
  We prove the first equality in~(2.17). To this end,
  fix $\alpha \in C^1([0,T])$ such that $\alpha(0) = 1$ and $\alpha(T) = 0$.
  Given any $\varphi \in V_0$,
  we insert $\Phi = \Phi(x,t) = \varphi(x) \, \alpha(t)$ ($(x,t) \in Q_T$)
  into~(2.9), multiply~(2.15) by $\alpha(t)$
  and integrate over $[0,T]$. It follows
  \begin{align*}
    (\varepsilon e_0, \varphi)_H
    & \, = \, - \int_0^T
    \big( (\varepsilon e)(t), \varphi \, \dot{\alpha}(t) \big){}_H \dt
    + \int_{Q_T}
    \big( {-}h \cdot \curl \varphi + j(e) \cdot \varphi
    \big) \, \alpha \dx \dt
    \\[\medskipamount]
    & \, = \, \langle (\widetilde{\varepsilon e})(0), \varphi \rangle_{V_0}
      \quad\text{(by~(2.13))}.
  \end{align*}
  Whence, $\varepsilon e_0 = (\widetilde{\varepsilon e})(0)$ in $V_0^*$.
  An analogous reasoning yields the second statement in~(2.17).

  Finally, estimates~(2.18) are readily deduced from~(2.15) and~(2.16).
  The proof of Theorem~2.1 is complete.
\end{proof}
\begin{corollary}
  Let hypotheses~\emph{(H1)--(H4)} hold
  and let $(e,h) \in L^2(Q_T)^3 \times L^2(Q_T)^3$
  be any weak solution of \emph{(1.1)--(1.4)}. Then,
  \begin{align}
    (\varepsilon e)' \in L^2(0,T;H)
    & \quad \Longleftrightarrow \quad
      h \in L^2(0,T;V);
      \tag{a}
    \\[\smallskipamount]
    (\mu h)' \in L^2(0,T;H)
    & \quad \Longleftrightarrow \quad
      e \in L^2(0,T;V_0).
      \tag{b}
  \end{align}
\end{corollary}
\begin{proof}[Proof of~\emph{(a)}]
  $(\Longrightarrow)$
  Assume $(\varepsilon e)'(t) \in H$ for some $t \in [0,T]$.
  We may further suppose that $j(e(\cdot,t)) \in H$
  and~(2.15) holds for the value $t$ under consideration.
  Thus, by~(2.11),
  \begin{equation*}
    \int_\Omega
    h(\cdot,t) \cdot \curl \varphi \dx
    \, = \, \int_\Omega
    \big( (\varepsilon e)'(t) + j(e(\cdot,t)) \big) \cdot \varphi \dx
    \quad\text{for all $\varphi \in C^\infty_c(\Omega)^3$}.
  \end{equation*}
  Whence, $h(\cdot,t) \in V$. A routine argument gives $h \in L^2(0,T;V)$.

  $(\Longleftarrow)$ Let $h \in L^2(0,T;V)$.
  Given any $\zeta \in C^\infty_c({]0,T[})$,
  we multiply~(2.15) by $\zeta(t)$
  and integrate over $t \in [0,T]$ to obtain
  \begin{align*}
    \Bigg( \int_0^T
    \dot{\zeta}(t) (\varepsilon e)(t) \dt, \varphi
    \Bigg)_{\!\! H}
    & = \, \Bigg\langle {-}\int_0^T
      \zeta(t) (\varepsilon e)'(t) \dt, \varphi
      \Bigg\rangle_{\!\! V_0}
    \\[\medskipamount]
    & = \, \Bigg( \int_0^T \zeta(t) \,
      \big( {-}\curl h(\cdot,t) + j(e(\cdot,t)) \big) \dt,
      \varphi \Bigg)_{\!\! H}
  \end{align*}
  for any $\varphi \in V_0$. Therefore,
  \begin{equation*}
    \int_0^T
    \dot{\zeta}(t) (\varepsilon e)(t) \dt
    \, = \, \int_0^T \zeta(t) \,
    \big( {-}\curl h(\cdot,t) + j(e(\cdot,t)) \big) \dt,
  \end{equation*}
  i.e.~$(\varepsilon e)' \in L^2(0,T;H)$.
\end{proof}
\begin{proof}[Proof of~\emph{(b)}]
  $(\Longrightarrow)$
  As above, assume $(\mu h)'(t) \in H$
  and~(2.16) holds for some $t \in [0,T]$. It follows
  \begin{equation*}
    \int_\Omega
    e(\cdot,t) \cdot \curl \psi \dx
    \, = \, -\int_\Omega
    (\mu h)'(t) \cdot \psi \dx
    \quad\text{for all $\psi \in V$}.
  \end{equation*}
  By~(2.8), $e(\cdot,t) \in V_0$.
  Again, by a routine argument we obtain $e \in L^2(0,T;V_0)$.
  
  The implication $(\Longleftarrow)$ can be proved
  by an argument that parallels item~(a).
\end{proof}
\section{Existence of $t$-continuous representatives \linebreak
  in the equivalence classes $e,h$}
\noindent
Besides~(H1), throughout the remainder of our paper we formulate
two more hypotheses for the matrices $\varepsilon(\cdot)$ and $\mu(\cdot)$:
\begin{equation}
  \text{\emph{$\varepsilon(x)$ and $\mu(x)$
      are symmetric for all $x \in \Omega$}};
  \tag{H5}
\end{equation}
\begin{equation}
  \left\{
    \begin{aligned}
      \, & \text{\emph{there exist constants $\varepsilon_* > 0$
          and $\mu_* > 0$ such that}}
      \\
      & \varepsilon(x) \xi \cdot \xi \ge \varepsilon_* |\xi|^2,
      \quad \mu(x) \xi \cdot \xi \ge \mu_* |\xi|^2
      \quad \text{\emph{for all $x \in \Omega$ and all $\xi \in \R^3$}}.
    \end{aligned}
    \tag{H6}
  \right.
\end{equation}
The following result is fundamental to our proof of the well-posedness
of~(1.1)--(1.4) in the $L^2$-setting.
\begin{theorem}
  Assume~\emph{(H1)--(H6)}.
  Then for every weak solution $(e,h) \in L^2(Q_T)^3 \times L^2(Q_T)^3$
  of~\emph{(1.1)--(1.4)} there exist representatives
  \begin{equation}
    \hat{e}, \hat{h} \in C([0,T];H)
  \end{equation}
  in the equivalence classes $e, h \in L^2(0,T;H)$
  \footnote{Remember the isometry $L^2(Q_T)^3 \cong L^2(0,T;H)$.},
  respectively, that satisfy the initial conditions
  \begin{equation}
    \hat{e}(0) = e_0,
    \quad \hat{h}(0) = h_0
    \quad\text{in $H$}.
  \end{equation}
\end{theorem}
We will prove this theorem via approximation of $(e,h)$ by time-averages.
This method has been used in~[19] for the proof
on integral estimates for functions on $Q_T$ (pp.~85--89)
and for proving an energy equality for weak solutions
of parabolic initial-boundary value problems (pp.~141--143)
as well as the continuity of these solutions in $t$
with respect to the $L^2(\Omega)$-norm (pp.~158--159).

The method of approximation of weak solutions of~(1.1)--(1.4)
by Steklov averages has been developed in~[25].
\hfill $\Box$
\subsection*{Preliminaries}
Let $f \in L^p(Q_T)$ ($1 \le p < \infty$).
We extend $f$ by zero for a.e.~$(x,t) \in \Omega \times (\R \setminus [0,T])$
and denote the function so defined a.e.~on $\Omega \times \R$ by $f$ again.
For $\lambda > 0$, define the \emph{Steklov averages of $f$}
for all $t \in [0,T]$ and a.e.~$x \in \Omega$ by
\begin{equation*}
  f_\lambda(x,t)
  \, = \, \frac{1}{\lambda} \int_t^{t + \lambda} f(x,s) \ds,
  \quad f_{\bar{\lambda}}(x,t)
  \, = \, \frac{1}{\lambda} \int_{t - \lambda}^t f(x,s) \ds,
\end{equation*}
(cf.~[19, p.~85, p.~141] ($p = 2$)).
We have
\begin{equation}
  \left\{
    \begin{aligned}
      \, & \text{\emph{for a.e.~$(x,t) \in Q_T$
          there exist the weak derivatives}}
      \\[\smallskipamount]
      & \partial_t f_\lambda(x,t)
      \, = \, \frac{1}{\lambda} \big( f(x,t + \lambda) - f(x,t) \big),
      \\[\smallskipamount]
      & \partial_t f_{\bar{\lambda}}(x,t)
      \, = \, \frac{1}{\lambda} \big( f(x,t) - f(x,t - \lambda) \big);
    \end{aligned}
  \right.
\end{equation}
\begin{equation}
  \left\{
    \begin{aligned}
      \, & \text{\emph{for any $\alpha \in C_c(\R)$}},
      \\
      & \int_{Q_T} f(x,t)
      \Bigg( \frac{1}{\lambda}
      \int_{t - \lambda}^t \alpha(s) \ds \Bigg) \dx \dt
      \, = \, \int_{Q_T} f_\lambda(x,t) \, \alpha(t) \dx \dt,
      \\
      & \int_{Q_T} f(x,t)
      \Bigg( \frac{1}{\lambda}
      \int_t^{t + \lambda} \alpha(s) \ds \Bigg) \dx \dt
      \, = \, \int_{Q_T} f_{\bar{\lambda}}(x,t) \, \alpha(t) \dx \dt;
    \end{aligned}
  \right.
\end{equation}
and
\begin{equation}
  f_\lambda \longrightarrow f
  \quad\text{\emph{and}}\quad f_{\bar{\lambda}} \longrightarrow f
  \quad\text{\emph{in $L^p(Q_T)$ as $\lambda \to 0$}}
\end{equation}
(see, e.g.~[25, Appendix~I, Prop.~I.1] for the proof
of~(3.3)--(3.5) for the Steklov average $f_\lambda$;
the same proofs work for $f_{\bar{\lambda}}$ with obvious changes).
\subsection*{Proof of Theorem~3.1}
Let $(e,h) \in L^2(Q_T) \times L^2(Q_T)$ be any weak solution of~(1.1)--(1.4).
Define $g(x,t) := j(x,t,e(x,t))$ for a.e.~$(x,t) \in Q_T$.
By~(H2), (H3), $g \in L^2(Q_T)$.

Fix real numbers $T_0, T_1$ such that
\begin{equation*}
  0 < T_0 < T_1 < T.
\end{equation*}
We divide the proof into three parts.
\subsubsection*{Part I. Integral identities for $(e_\lambda,h_\lambda)$
and $(e_{\bar{\lambda}},h_{\bar{\lambda}})$}
\begin{lemma}[Integral identities for $(e_\lambda,h_\lambda)$]
  For every $0 < \lambda < T - T_1$,
  \begin{align}
    & \left\{
      \begin{aligned}
        \; & \int_\Omega
        \big( \partial_t (\varepsilon e)_\lambda(x,t) \cdot \varphi(x)
        - h_\lambda(x,t) \cdot \curl \varphi(x)
        + g_\lambda(x,t) \cdot \varphi(x) \big) \dx
        \, = \, 0
        \\
        & \, \text{for a.e.~$t \in [0,T_1]$ and all $\varphi \in V_0$},
      \end{aligned}
          \right.
    \\[\medskipamount]
    & \left\{
      \begin{aligned}
        \; & \int_\Omega
        \big( \partial_t (\mu h)_\lambda(x,t) \cdot \psi(x)
        + e_\lambda(x,t) \cdot \curl \psi(x) \big) \dx
        \, = \, 0
        \\
        & \, \text{for a.e.~$t \in [0,T_1]$ and all $\psi \in V$}.
      \end{aligned}
          \right.
  \end{align}
  Moreover, for a.e.~$t \in [0,T_1]$,
  \begin{gather}
    e_\lambda(\cdot,t) \in V_0,
    \quad h_\lambda(\cdot,t) \in V,
    \\[\medskipamount]
    \int_\Omega
    (\curl e_\lambda(x,t)) \cdot h_\lambda(x,t) \dx
    \, = \, \int_\Omega
    e_\lambda(x,t) \cdot \curl h_\lambda(x,t) \dx.
  \end{gather}
\end{lemma}
\begin{proof}
  Let $\alpha \in C_c(\R)$ be such that $\supp(\alpha) \subset {]0,T_1[}$.
  Given $\varphi \in V_0$, we consider the function
  \begin{equation*}
    \Phi(x,t) \, = \, \varphi(x) \int_{t - \lambda}^t \alpha(s) \ds
    \quad \text{for a.e.~$(x,t) \in Q_T$}.
  \end{equation*}
  Then
  \begin{gather*}
    \Phi(\cdot \,,t) \in V_0
    \quad \text{for all $t \in [0,T]$},
    \quad \Phi(x,0) \, = \, \Phi(x,T) \, = \, 0
    \quad \text{for a.e.~$x \in \Omega$},
    \\[\smallskipamount]
    \partial_t \Phi(x,t)
    \, = \, \varphi(x) \big( \alpha(t) - \alpha(t - \lambda) \big)
    \quad \text{for a.e.~$(x,t) \in Q_T$},
  \end{gather*}
  i.e., $\Phi$ is an admissible test function in~(2.9). It follows
  \begin{multline*}
    \int_{Q_T}
    \big( (\varepsilon e)(x,t + \lambda) - (\varepsilon e)(x,t) \big)
    \cdot \varphi(x) \, \alpha(t) \dx \dt
    \\
    = \, \int_{Q_T}
    \big( h(x,t) \cdot \curl \varphi(x)
    - g(x,t) \cdot \varphi(x) \big)
    \Bigg( \; \int_{t - \lambda}^t \alpha(s) \ds \Bigg) \dx \dt.
  \end{multline*}
  We divide each term of this equation by $\lambda$
  and make use of~(3.3) and~(3.4) for $f_\lambda$
  ($f = \varepsilon e$, resp.~$f = h \cdot \curl \varphi$,
  $f = g \cdot \varphi$) to obtain
  \begin{equation*}
    \int_{Q_T}
    \big( \partial_t (\varepsilon e)_\lambda (x,t) \cdot \varphi(x)
    - h_\lambda(x,t) \cdot \curl \varphi(x)
    + g_\lambda(x,t) \cdot \varphi(x) \big) \, \alpha(t) \dx \dt
    \, = \, 0.
  \end{equation*}
  The claim~(3.6) follows from this equation by a routine argument.
  We note that the set of measure zero of those $t \in [0,T_1]$
  for which~(3.6) fails, may depend on $\lambda$
  but is independent of $\varphi \in V_0$.

  Next, given $\psi \in V$, the function
  \begin{equation*}
    \Psi(x,t) \, = \, \psi(x) \int_{t - \lambda}^t \alpha(s) \ds
    \quad \text{for a.e.~$(x,t) \in Q_T$}
  \end{equation*}
  is an admissible test function in~(2.10).
  Then one obtains~(3.7) by analogous arguments as for the proof of~(3.6)
  (make use of~(3.3) and~(3.4) for $f_\lambda$
  ($f = \mu h$, resp. $f = e \cdot \curl \psi$)).

  We prove $e_\lambda(\cdot,t) \in V_0$ for a.e.~$t \in [0,T_1]$
  such that~(3.7) holds. Indeed, for any of these values of $t$, we have
  \begin{equation*}
    \int_\Omega e_\lambda(x,t) \cdot \curl \psi \dx
    \, = \, - \int_\Omega \partial_t (\mu h)_\lambda(x,t) \cdot \psi \dx
    \quad\text{for all $\psi \in V$}.
  \end{equation*}
  Observing that
  \begin{equation*}
    \partial_t (\mu h)_\lambda(\cdot,t)
    \, = \, \frac{1}{\lambda}
    \big( (\mu h)(\cdot,t + \lambda) - (\mu h)(\cdot,t) \big)
    \in L^2(\Omega)^3,
  \end{equation*}
  it follows $e_\lambda(\cdot,t) \in V_0$ (see~(2.8)).
  
  To see $h_\lambda(\cdot,t) \in V$ for a.e.~$t \in [0,T_1]$,
  it suffices to note that
  \begin{equation*}
    \partial_t (\varepsilon e)_\lambda(\cdot,t)
    + g_\lambda(\cdot,t)
    \, = \, \frac{1}{\lambda}
    \big( (\varepsilon e)(\cdot,t + \lambda) - (\varepsilon e)(\cdot,t) \big)
    + g_\lambda(\cdot,t)
    \in L^2(\Omega)^3
  \end{equation*}
  and that~(3.6) evidently holds for all $\varphi \in C^\infty_c(\Omega)^3$.
  Whence, the claim~(3.8).

  Finally, (3.9) is a consequence of~(3.8)
  and our definition of the space~$V_0$.
\end{proof}
\begin{lemma}[Integral identities for $(e_{\bar{\lambda}},h_{\bar{\lambda}})$]
  For every $0 < \lambda < T_0$,
  \begin{align}
    & \left\{
      \begin{aligned}
        \; & \int_\Omega
        \big( \partial_t (\varepsilon e)_{\bar{\lambda}}(x,t) \cdot \varphi(x)
        - h_{\bar{\lambda}}(x,t) \cdot \curl \varphi(x)
        + g_{\bar{\lambda}}(x,t) \cdot \varphi(x) \big) \dx
        \, = \, 0
        \\
        & \, \text{for a.e.~$t \in [T_0,T]$ and all $\varphi \in V_0$},
      \end{aligned}
          \right.
    \\[\medskipamount]
    & \left\{
      \begin{aligned}
        \; & \int_\Omega
        \big( \partial_t (\mu h)_{\bar{\lambda}}(x,t) \cdot \psi(x)
        + e_{\bar{\lambda}}(x,t) \cdot \curl \psi(x) \big) \dx
        \, = \, 0
        \\
        & \, \text{for a.e.~$t \in [T_0,T]$ and all $\psi \in V$}.
      \end{aligned}
          \right.
  \end{align}
  Moreover, for a.e.~$t \in [T_0,T]$,
  \begin{gather}
    e_{\bar{\lambda}}(\cdot,t) \in V_0,
    \quad h_{\bar{\lambda}}(\cdot,t) \in V,
    \\[\medskipamount]
    \int_\Omega
    (\curl e_{\bar{\lambda}}(x,t)) \cdot h_{\bar{\lambda}}(x,t) \dx
    \, = \, \int_\Omega
    e_{\bar{\lambda}}(x,t) \cdot \curl h_{\bar{\lambda}}(x,t) \dx.
  \end{gather}
\end{lemma}
\begin{proof}
  Let $\alpha \in C_c(\R)$ be such that $\supp(\alpha) \subset {]T_0,T[}$.
  Given $\varphi \in V_0$, we consider the function
  \begin{equation*}
    \Phi(x,t) \, = \, \varphi(x) \int_t^{t + \lambda} \alpha(s) \ds
    \quad \text{for a.e.~$(x,t) \in Q_T$}.
  \end{equation*}
  Then
  \begin{gather*}
    \Phi(\cdot \,,t) \in V_0
    \quad \text{for all $t \in [0,T]$},
    \quad \Phi(x,0) \, = \, \Phi(x,T) \, = \, 0
    \quad \text{for a.e.~$x \in \Omega$},
    \\[\medskipamount]
    \partial_t \Phi(x,t)
    \, = \, \varphi(x) \big( \alpha(t + \lambda) - \alpha(t) \big)
    \quad \text{for a.e.~$(x,t) \in Q_T$},
  \end{gather*}
  i.e.~$\Phi$ is an admissible test function in~(2.9). It follows
  \begin{multline*}
    \int_{Q_T}
    \big( (\varepsilon e)(x,t) - (\varepsilon e)(x,t - \lambda) \big)
    \cdot \varphi(x) \, \alpha(t) \dx \dt
    \\
    = \, \int_{Q_T}
    \big( h(x,t) \cdot \curl \varphi(x)
    - g(x,t) \cdot \varphi(x) \big)
    \Bigg( \int_t^{t + \lambda} \alpha(s) \ds \Bigg) \dx \dt.
  \end{multline*}
  We divide each term of this equation by $\lambda$ and make use
  of~(3.3) and~(3.4) for $f_{\bar{\lambda}}$
  ($f = \varepsilon e$, resp. $f = h \cdot \curl \varphi$,
  $f = g \cdot \varphi$) to obtain
  \begin{equation*}
    \int_{Q_T}
    \big( \partial_t (\varepsilon e)_{\bar{\lambda}} (x,t) \cdot \varphi(x)
    - h_{\bar{\lambda}}(x,t) \cdot \curl \varphi(x)
    + g_{\bar{\lambda}}(x,t) \cdot \varphi(x) \big) \, \alpha(t) \dx \dt
    \, = \, 0.
  \end{equation*}
  The claim~(3.10) follows from this equation by a routine argument.
  We note that the set of measure zero of those $t \in [T_0,T]$
  for which~(3.10) fails, may depend on $\lambda$
  but is independent of $\varphi \in V_0$.
  
  Next, given $\psi \in V$, the function
  \begin{equation*}
    \Psi(x,t) \, = \, \psi(x) \int_t^{t + \lambda} \alpha(s) \ds
    \quad \text{for a.e.~$(x,t) \in Q_T$}
  \end{equation*}
  is an admissible test function in~(2.10).
  Then one obtains~(3.11) by analogous arguments as for the proof of~(3.10)
  (make use of~(3.3) and~(3.4) for $f_{\bar{\lambda}}$
  ($f = \mu h$, resp. $f = e \cdot \curl \psi$)).

  We prove $e_{\bar{\lambda}}(\cdot,t) \in V_0$ for a.e.~$t \in [T_0,T]$
  such that~(3.11) holds. Indeed, for any of these values of $t$, we have
  \begin{equation*}
    \int_\Omega e_{\bar{\lambda}}(x,t) \cdot \curl \psi \dx
    \, = \, - \int_\Omega \partial_t (\mu h)_{\bar{\lambda}}(x,t) \cdot \psi \dx
    \quad\text{for all $\psi \in V$}.
  \end{equation*}
  Observing that
  \begin{equation*}
    \partial_t (\mu h)_{\bar{\lambda}}(\cdot,t)
    \, = \, \frac{1}{\lambda}
    \big( (\mu h)(\cdot,t) - (\mu h)(\cdot,t - \lambda) \big)
    \in L^2(\Omega)^3,
  \end{equation*}
  it follows $e_{\bar{\lambda}}(\cdot,t) \in V_0$ (see~(2.8)).
  
  To see $h_{\bar{\lambda}}(\cdot,t) \in V$ for a.e.~$t \in [T_0,T]$,
  it suffices to note that
  \begin{equation*}
    \partial_t (\varepsilon e)_{\bar{\lambda}}(\cdot,t)
    + g_{\bar{\lambda}}(\cdot,t)
    \, = \, \frac{1}{\lambda}
    \big( (\varepsilon e)(\cdot,t) - (\varepsilon e)(\cdot,t - \lambda) \big)
    + g_{\bar{\lambda}}(\cdot,t)
    \in L^2(\Omega)^3
  \end{equation*}
  and that~(3.10) evidently holds for all $\varphi \in C^\infty_c(\Omega)^3$.
  Whence, the claim~(3.12).

  Finally, (3.13) is a consequence of~(3.12)
  and our definition of the space~$V_0$.  
\end{proof}
\subsubsection*{Part II. Estimates for the differences
  of $(e_\lambda,h_\lambda)$ and of $(e_{\bar{\lambda}},h_{\bar{\lambda}})$}
Let $(\lambda_m)_{m \in \N}$ be any sequence of real numbers such that
$0 < \lambda_m < \min\{T_0,T - T_1\}$ for all $m \in \N$,
and $\lambda_m \to 0$ as $m \to \infty$.
Following ideas from~[19, pp.~158--159],
we establish estimates for the differences
$e_{\lambda_m} - e_{\lambda_n}$, $h_{\lambda_m} - h_{\lambda_n}$ and
$e_{\bar{\lambda}_m} - e_{\bar{\lambda}_n}$,
$h_{\bar{\lambda}_m} - h_{\bar{\lambda}_n}$
which enable us to prove that
$(e_{\lambda_m}){}_{m \in \N}$, $(h_{\lambda_m}){}_{m \in \N}$
and $(e_{\bar{\lambda}_m}){}_{m \in \N}$, $(h_{\bar{\lambda}_m}){}_{m \in \N}$
are Cauchy sequences in $C([0,T];H)$.
Here, crucial points are the identities~(3.9) and~(3.13)
that we may use as well for the differences above.
Moreover, applying the distributional derivatives of $e_\lambda$, $h_\lambda$
makes our presentation simpler than the one in~[19].
\hfill $\Box$

To simplify the following discussion,
we introduce the weighted scalar products on $H$
\begin{equation*}
  (u,v)_{H_\varepsilon}
  \, := \, \int_\Omega \varepsilon(x) \, u(x) \cdot v(x) \dx, \quad
  (u,v)_{H_\mu}
  \, := \, \int_\Omega \mu(x) \, u(x) \cdot v(x) \dx.
\end{equation*}
Both scalar products are equivalent to the standard scalar product on $H$.

We consider~(3.6) and~(3.7)
with $\lambda = \lambda_m$ and $\lambda = \lambda_n$,
form differences $e_{\lambda_m} - e_{\lambda_n}$
and $h_{\lambda_m} - h_{\lambda_n}$, take then
$\varphi = e_{\lambda_m}(\cdot,t) - e_{\lambda_n}(\cdot,t)$ in~(3.6)
and $\psi = h_{\lambda_m}(\cdot,t) - h_{\lambda_n}(\cdot,t)$ in~(3.7),
add the identities so obtained (cf.~[19, p.~159])
and observe~(3.9) with $e_{\lambda_m} - e_{\lambda_n}$,
$h_{\lambda_m} - h_{\lambda_n}$ in place of $e_\lambda$, $h_\lambda$.
This gives
\begin{equation}
  \left\{
    \begin{aligned}
      \; & \frac{\mathrm{d}}{\mathrm{d} \tau}
      \Big( \| e_{\lambda_m}(\tau) - e_{\lambda_n}(\tau) \|^2_{H_\varepsilon}
      + \| h_{\lambda_m}(\tau) - h_{\lambda_n}(\tau) \|^2_{H_\mu} \Big)
      \\[\smallskipamount]
      & = \, -2 \, \big( g_{\lambda_m}(\tau) - g_{\lambda_n}(\tau),
      e_{\lambda_m}(\tau) - e_{\lambda_n}(\tau) \big){}_H
      \\[\medskipamount]
      & \text{for all $m, n \in \N$ and a.e.~$\tau \in [0,T_1]$}.
    \end{aligned}
  \right.
\end{equation}
\begin{lemma}
  For all $m$, $n \in \N$ and all $t \in [0,T_1]$,
  \begin{equation}
    \left\{
      \begin{aligned}
        \; &
        T_1 \Big( \| e_{\lambda_m}(t) - e_{\lambda_n}(t) \|^2_{H_\varepsilon}
        + \| h_{\lambda_m}(t) - h_{\lambda_n}(t) \|^2_{H_\mu} \Big)
        \\[\bigskipamount]
        & = \, \int_0^{T_1}
        \Big( \| e_{\lambda_m}(s) - e_{\lambda_n}(s) \|^2_{H_\varepsilon}
        + \| h_{\lambda_m}(s) - h_{\lambda_n}(s) \|^2_{H_\mu} \Big) \ds
        \\
        & \quad - 2 \int_0^{T_1}
        \Bigg( \int_s^t 
        \big( g_{\lambda_m}(\tau) - g_{\lambda_n}(\tau),
        e_{\lambda_m}(\tau) - e_{\lambda_n}(\tau) \big){}_H \dtau
        \Bigg) \ds.
        \; \footnotemark
      \end{aligned}
    \right.
  \end{equation}
  \footnotetext{For $s$, $t \in [0,T_1]$, $s > t$, define
    $\int_s^t \beta(\tau) \dtau = - \int_t^s \beta(\tau) \dtau$.}
\end{lemma}
\begin{proof}
  Let $t \in {]0,T_1[}$. Firstly, given any $s \in [0,t]$,
  we integrate~(3.14) over the interval $[s,t]$ to obtain
  \begin{align*}
    \| e_{\lambda_m}(t) - e_{\lambda_n}(t) \|^2_{H_\varepsilon}
    + \| h_{\lambda_m}(t) - h_{\lambda_n}(t) \|^2_{H_\mu}
    \, & = \, \| e_{\lambda_m}(s) - e_{\lambda_n}(s) \|^2_{H_\varepsilon}
         + \| h_{\lambda_m}(s) - h_{\lambda_n}(s) \|^2_{H_\mu}
    \\
       & - 2 \int_s^t
         \big( g_{\lambda_m}(\tau) - g_{\lambda_n}(\tau),
         e_{\lambda_m}(\tau) - e_{\lambda_n}(\tau) \big){}_H \dtau.
  \end{align*}
  We now integrate this equation
  with respect to the variable $s$ over the interval $[0,t]$. It follows
  \begin{equation}
    \left\{
      \begin{aligned}
        \; &
        t \, \Big( \| e_{\lambda_m}(t) - e_{\lambda_n}(t) \|^2_{H_\varepsilon}
        + \| h_{\lambda_m}(t) - h_{\lambda_n}(t) \|^2_{H_\mu} \Big)
        \\[\bigskipamount]
        & = \, \int_0^t
        \Big( \| e_{\lambda_m}(s) - e_{\lambda_n}(s) \|^2_{H_\varepsilon}
        + \| h_{\lambda_m}(s) - h_{\lambda_n}(s) \|^2_{H_\mu} \Big) \ds
        \\
        & \quad - 2 \int_0^t
        \Bigg( \int_s^t 
        \big( g_{\lambda_m}(\tau) - g_{\lambda_n}(\tau),
        e_{\lambda_m}(\tau) - e_{\lambda_n}(\tau) \big){}_H \dtau
        \Bigg) \ds.
      \end{aligned}
    \right.
  \end{equation}

  Secondly, given any $s \in [t,T_1]$,
  we integrate~(3.14) over the interval $[t,s]$ to get
  \begin{align*}
    \| e_{\lambda_m}(t) - e_{\lambda_n}(t) \|^2_{H_\varepsilon}
    + \| h_{\lambda_m}(t) - h_{\lambda_n}(t) \|^2_{H_\mu}
    \, & = \, \| e_{\lambda_m}(s) - e_{\lambda_n}(s) \|^2_{H_\varepsilon}
         + \| h_{\lambda_m}(s) - h_{\lambda_n}(s) \|^2_{H_\mu}
    \\
       & + 2 \int_t^s
         \big( g_{\lambda_m}(\tau) - g_{\lambda_n}(\tau),
         e_{\lambda_m}(\tau) - e_{\lambda_n}(\tau) \big){}_H \dtau.
  \end{align*}
  We integrate this equation
  with respect to the variable $s$ over the interval $[t,T_1]$. This yields
  \begin{equation}
    \left\{
      \begin{aligned}
        \; & (T_1 - t) \,
        \Big( \| e_{\lambda_m}(t) - e_{\lambda_n}(t) \|^2_{H_\varepsilon}
        + \| h_{\lambda_m}(t) - h_{\lambda_n}(t) \|^2_{H_\mu} \Big)
        \\[\bigskipamount]
        & = \, \int_t^{T_1}
        \Big( \| e_{\lambda_m}(s) - e_{\lambda_n}(s) \|^2_{H_\varepsilon}
        + \| h_{\lambda_m}(s) - h_{\lambda_n}(s) \|^2_{H_\mu} \Big) \ds
        \\
        & \quad + 2 \int_t^{T_1}
        \Bigg( \int_t^s 
        \big( g_{\lambda_m}(\tau) - g_{\lambda_n}(\tau),
        e_{\lambda_m}(\tau) - e_{\lambda_n}(\tau) \big){}_H \dtau
        \Bigg) \ds.
      \end{aligned}
    \right.
  \end{equation}
  Finally, if $t = 0$ or $t = T_1$, then~(3.16) resp.~(3.17) are trivial.
  Adding~(3.16) and~(3.17) we obtain~(3.15) for all $t \in [0,T_1]$.
\end{proof}

We finish Part~II with an analogue of Lemma~3.3.
For this we consider integral identitities~(3.10) and~(3.11),
and repeat the arguments which led to~(3.14).
Using~(3.13) with $e_{\bar{\lambda}_m} - e_{\bar{\lambda}_n}$,
$h_{\bar{\lambda}_m} - h_{\bar{\lambda}_n}$
instead of $e_{\bar{\lambda}}$, $h_{\bar{\lambda}}$, one obtains
\begin{equation}
  \left\{
    \begin{aligned}
      \; & \frac{\mathrm{d}}{\mathrm{d} \tau}
      \Big( \| e_{\bar{\lambda}_m}(\tau)
      - e_{\bar{\lambda}_n}(\tau) \|^2_{H_\varepsilon}
      + \| h_{\bar{\lambda}_m}(\tau)
      - h_{\bar{\lambda}_n}(\tau) \|^2_{H_\mu} \Big)
      \\[\smallskipamount]
      & = \, -2 \, \big( g_{\bar{\lambda}_m}(\tau)
      - g_{\bar{\lambda}_n}(\tau),
      e_{\bar{\lambda}_m}(\tau)
      - e_{\bar{\lambda}_n}(\tau) \big){}_H
      \\[\medskipamount]
      & \text{for all $m, n \in \N$ and a.e.~$\tau \in [T_0,T]$}.
    \end{aligned}
  \right.
\end{equation}
\begin{lemma}
  For all $m$, $n \in \N$ and all $t \in [T_0,T]$,
  \begin{equation}
    \left\{
      \begin{aligned}
        \; &
        (T - T_0) \Big( \| e_{\bar{\lambda}_m}(t)
        - e_{\bar{\lambda}_n}(t) \|^2_{H_\varepsilon}
        + \| h_{\bar{\lambda}_m}(t)
        - h_{\bar{\lambda}_n}(t) \|^2_{H_\mu} \Big)
        \\[\bigskipamount]
        & = \, \int_{T_0}^T
        \Big( \| e_{\bar{\lambda}_m}(s)
        - e_{\bar{\lambda}_n}(s) \|^2_{H_\varepsilon}
        + \| h_{\bar{\lambda}_m}(s)
        - h_{\bar{\lambda}_n}(s) \|^2_{H_\mu} \Big) \ds
        \\
        & \quad - 2 \int_{T_0}^T
        \Bigg( \int_s^t 
        \big( g_{\bar{\lambda}_m}(\tau)
        - g_{\bar{\lambda}_n}(\tau),
        e_{\bar{\lambda}_m}(\tau)
        - e_{\bar{\lambda}_n}(\tau) \big){}_H \dtau
        \Bigg) \ds.
      \end{aligned}
    \right.
  \end{equation}
\end{lemma}
\begin{proof}
  Let $t \in {]T_0,T[}$. Firstly, given any $s \in [T_0,t]$,
  we integrate~(3.18) over the interval $[s,t]$
  and integrate then the equation so obtained 
  with respect to the variable $s$ over the interval $[T_0,t]$. This gives
  \begin{equation}
    \left\{
      \begin{aligned}
        \; &
        (t - T_0) \Big( \| e_{\bar{\lambda}_m}(t)
        - e_{\bar{\lambda}_n}(t) \|^2_{H_\varepsilon}
        + \| h_{\bar{\lambda}_m}(t)
        - h_{\bar{\lambda}_n}(t) \|^2_{H_\mu} \Big)
        \\[\bigskipamount]
        & = \, \int_{T_0}^t
        \Big( \| e_{\bar{\lambda}_m}(s)
        - e_{\bar{\lambda}_n}(s) \|^2_{H_\varepsilon}
        + \| h_{\bar{\lambda}_m}(s)
        - h_{\bar{\lambda}_n}(s) \|^2_{H_\mu} \Big) \ds
        \\
        & \quad - 2 \int_{T_0}^t
        \Bigg( \int_s^t 
        \big( g_{\bar{\lambda}_m}(\tau)
         - g_{\bar{\lambda}_n}(\tau),
         e_{\bar{\lambda}_m}(\tau)
         - e_{\bar{\lambda}_n}(\tau) \big){}_H \dtau
         \Bigg) \ds.
      \end{aligned}
    \right.
  \end{equation}
  
  Secondly, given any $s \in [t,T]$,
  we integrate~(3.18) over the interval $[t,s]$
  and integrate then the equation obtained in this way
  with respect to the variable $s$ over the interval $[t,T]$ to find
  \begin{equation}
    \left\{
      \begin{aligned}
        \; & (T - t) \,
        \Big( \| e_{\bar{\lambda}_m}(t)
        - e_{\bar{\lambda}_n}(t) \|^2_{H_\varepsilon}
        + \| h_{\bar{\lambda}_m}(t)
        - h_{\bar{\lambda}_n}(t) \|^2_{H_\mu} \Big)
        \\[\bigskipamount]
        & = \, \int_t^T
        \Big( \| e_{\bar{\lambda}_m}(s)
        - e_{\bar{\lambda}_n}(s) \|^2_{H_\varepsilon}
        + \| h_{\bar{\lambda}_m}(s)
        - h_{\bar{\lambda}_n}(s) \|^2_{H_\mu} \Big) \ds
        \\
        & \quad + 2 \int_t^T
        \Bigg( \int_t^s 
        \big( g_{\bar{\lambda}_m}(\tau)
         - g_{\bar{\lambda}_n}(\tau),
         e_{\bar{\lambda}_m}(\tau)
         - e_{\bar{\lambda}_n}(\tau) \big){}_H \dtau        
        \Bigg) \ds.
      \end{aligned}
    \right.
  \end{equation}
  Finally, if $t = T_0$ or $t = T$, then~(3.20) resp.~(3.21) are trivial.
  Adding~(3.20) and~(3.21) we obtain~(3.19) for all $t \in [T_0,T]$.
\end{proof}
\subsubsection*{Part III. Proof of Theorem~3.1 completed}
Let $(\lambda_m)_{m \in \N}$ be any sequence of real numbers
as at the beginning of Part~II. 
From~(3.15) we infer
\begin{align*}
  & \max_{t \in [0,T_1]}
    \|e_{\lambda_m}(t) - e_{\lambda_n}(t)\|^2_{H_\varepsilon}
    + \max_{t \in [0,T_1]}
    \|h_{\lambda_m}(t) - h_{\lambda_n}(t)\|^2_{H_\mu}
  \\[\medskipamount]
  & = \, \frac{1}{T_1} \int_0^{T_1}
    \Big( \| e_{\lambda_m}(s) - e_{\lambda_n}(s) \|^2_{H_\varepsilon}
    + \| h_{\lambda_m}(s) - h_{\lambda_n}(s) \|^2_{H_\mu} \Big) \ds
  \\
  & \quad + 2 \int_0^{T_1}
    \| g_{\lambda_m}(\tau) - g_{\lambda_n}(\tau) \|_H \,
    \| e_{\lambda_m}(\tau) - e_{\lambda_n}(\tau)\|_H \dtau
\end{align*}
for all $m$, $n \in \N$. Observing~(H6) and~(3.5) we see that
$(e_{\lambda_m})_{m \in \N}$, $(h_{\lambda_m})_{m \in \N}$
are Cauchy sequences in $C([0,T_1];H)$.
Analogously, (3.19) implies that
$(e_{\bar{\lambda}_m})_{m \in \N}$, $(h_{\bar{\lambda}_m})_{m \in \N}$
are Cauchy sequences in $C([T_0,T];H)$.
Thus, there exist
\begin{equation*}
  \underline{e}, \underline{h} \in C([0,T_1];H)
  \quad\text{and}\quad
  \overline{e}, \overline{h} \in C([T_0,T];H)
\end{equation*}
such that
\begin{align}
  e_{\lambda_m} \longrightarrow \, \underline{e}
  \quad\text{and}\quad h_{\lambda_m} \longrightarrow \, \underline{h}
  & \quad\text{in $C([0,T_1];H)$},
  \\
  e_{\bar{\lambda}_m} \longrightarrow \, \overline{e}
  \quad\text{and}\quad h_{\bar{\lambda}_m} \longrightarrow \, \overline{h}
  & \quad\text{in $C([T_0,T];H)$}
\end{align}
as $m \to \infty$. A routine argument gives
\begin{alignat*}{3}
  \underline{e}(t)
  & \, = \, e(t),
  & \quad \underline{h}(t)
  & \, = \, h(t)
  & \quad & \text{in $H$ for a.e.~$t \in [0,T_1]$},
  \\
  \overline{e}(t)
  & \, = \, e(t),
  & \quad \overline{h}(t)
  & \, = \, h(t)
  & \quad & \text{in $H$ for a.e.~$t \in [T_0,T]$}.
\end{alignat*}

Put $T_* = \frac{1}{2} (T_0 + T_1)$ and define
\begin{equation*}
  \hat{e}(t)
  \, := \,
  \begin{cases}
    \underline{e}(t) & \text{if $t \in [0,T_*]$}, \\
    \overline{e}(t) & \text{if $t \in [T_*,T]$};
  \end{cases}
  \quad
  \hat{h}(t)
  \, := \,
  \begin{cases}
    \underline{h}(t) & \text{if $t \in [0,T_*]$}, \\
    \overline{h}(t) & \text{if $t \in [T_*,T]$}.
  \end{cases}
\end{equation*}
We obtain
\begin{equation}
  \hat{e}, \hat{h} \in C([0,T];H),
  \quad \hat{e}(t) = e(t),
  \quad \hat{h}(t) = h(t)
  \quad \text{in $H$ for a.e.~$t \in [0,T]$},
\end{equation}
i.e.~(3.1) holds.

It remains to prove $\hat{e}(0) = e_0$ in $H$ (cf.~(3.2)). The proof
of $\hat{h}(0) = h_0$ follows the same lines with minor modifications.
Identifying $\varepsilon \hat{e} \in C([0,T];H)$ with an element
in $C([0,T];V_0^*)$ it follows
\begin{equation*}
  (\varepsilon \hat{e})(t)
  \, = \, (\widetilde{\varepsilon e})(t)
  \quad\text{in $V_0^*$ for all $t \in [0,T]$},
\end{equation*}
where $\widetilde{\varepsilon e}:[0,T] \longrightarrow V_0^*$
denotes the absolutely continuous representative
in the equivalence class $\varepsilon e \in L^2(0,T;V_0^*)$ (cf.~Thm.~2.1).
Thus, for all $\varphi \in V_0$,
\begin{equation*}
  (\varepsilon \hat{e}(0), \varphi)_H
  \, = \, \langle (\varepsilon \hat{e})(0), \varphi \rangle_{V_0}
  \, = \, \langle (\widetilde{\varepsilon e})(0), \varphi \rangle_{V_0}
  \, = \, \langle \varepsilon e_0, \varphi \rangle_{V_0}
  \, = \, (\varepsilon e_0, \varphi)_H.
\end{equation*}
The proof of Theorem~3.1 is complete.
\section{Energy equality. Well-posedness of (1.1)--(1.4)}
\noindent
In this section, we prove that under the hypotheses~(H1)--(H6)
\emph{any} weak solution of~(1.1)--(1.4) obeys an energy equality.
If, in addition, $\xi \longmapsto j(\cdot,\cdot,\xi)$ is monotone,
then the well-posedness of~(1.1)--(1.4) in the framework of~$L^2$
is easily derived from the energy equality.

Besides its independent interest, this equality is fundamental to our proof
of the existence of a weak solution of~(1.1)--(1.4)
via the Faedo-Galerkin method (see Section~5).

The following theorem is the main result of our paper.
\begin{theorem}[Energy equality]
  Assume~\emph{(H1)--(H6)}.
  Let $(e,h) \in L^2(Q_T)^3 \times L^2(Q_T)^3$ be any weak solution
  of~\emph{(1.1)--(1.4)} and denote by
  \begin{equation*}
    \hat{e}, \hat{h} \in C([0,T];H)
  \end{equation*}
  the continuous representatives in the equivalence classes $e, h$
  \emph{(}cf.~Theorem~3.1\,\emph{)}.
  Then,
  \begin{equation}
    \frac{1}{2}
    \Big( \big\| \hat{e}(t) \big\|{}^2_{H_\varepsilon}
    + \big\| \hat{h}(t) \big\|{}^2_{H_\mu} \Big)
    + \int_0^t (j(e),e)_H \ds
    \, = \, \frac{1}{2}
    \Big( \| e_0 \|^2_{H_\varepsilon} + \| h_0 \|^2_{H_\mu} \Big)
    \quad\text{for all $t \in [0,T]$}.
  \end{equation}
\end{theorem}
\begin{proof}
  For notational simplicity, we write
  \begin{equation*}
    \hat{\E}(t)
    \, = \, \frac{1}{2}
    \Big( \big\| \hat{e}(t) \big\|{}^2_{H_\varepsilon}
    + \big\| \hat{h}(t) \big\|{}^2_{H_\mu} \Big),
    \quad t \in [0,T]
  \end{equation*}
  (cf.~(1.7); remember $\hat{e}(0) = e_0$, $\hat{h}(0) = h_0$).

  As in Section~3, let $T_0$, $T_1$ be two real numbers
  such that $0 < T_0 < T_1 < T$, and let $0 < \lambda < \min \{T_0, T-T_1\}$.
  From Lemma~3.1 it follows that
  \begin{equation}
    \frac{1}{2} \Big( \| e_\lambda(t) \|^2_{H_\varepsilon}
    + \| h_\lambda(t) \|^2_{H_\mu} \Big)
    + \int_0^t (g_\lambda,e_\lambda)_H \ds
    \, = \, \frac{1}{2} \Big( \| e_\lambda(0) \|^2_{H_\varepsilon}
    + \| h_\lambda(0) \|^2_{H_\mu} \Big)
  \end{equation}
  for all $t \in [0,T_1]$
  ($g = j(\cdot,\cdot,e)$; cf.~the proof of Theorem~3.1).

  Let $(\lambda_m)_{m \in \N}$ be any sequence of real numbers
  such that $0 < \lambda_m < \min \{T_0, T-T_1\}$ for all $m \in \N$,
  and $\lambda_m \to 0$ as $m \to \infty$
  (cf.~the proof of Theorem~3.1, Part~II).
  Taking $\lambda = \lambda_m$ in~(4.2) and observing~(3.22) and~(3.24)
  we obtain upon letting tend $m \to \infty$ in~(4.2) the equality
  \begin{equation}
    \hat{\E}(t)
    + \int_0^t (j(e),e)_H \ds
    \, = \, \hat{\E}(0)
    \quad\text{for all $t \in [0,T_1]$}.
  \end{equation}

  Next, using~Lemma~3.2 we find by an analogous reasoning
  (this time by the aid of~(3.23) and~(3.24))
  \begin{equation*}
    \hat{\E}(t)
    + \int_{T_0}^t (j(e),e)_H \ds
    \, = \, \hat{\E}(T_0)
    \quad\text{for all $t \in [T_0,T]$}.
  \end{equation*}
  It follows that, for all $t \in [T_0,T]$,
  \begin{align*}
    \hat{\E}(t)
    + \int_0^t (j(e),e)_H \ds
    \, & = \, \hat{\E}(T_0)
         + \int_0^{T_0} (j(e),e)_H \ds
    \\
    \, & = \, \hat{\E}(0)
         \quad\text{(by~(4.3))}.
  \end{align*}
  Whence, (4.1).
\end{proof}
\begin{remark}
  In his seminal paper~[13], \textsc{K.~O.~Friedrichs}
  developed a theory of weak solutions for a large class
  of initial-boundary value problems for symmetric linear hyperbolic systems
  where he made use of energy integral identities.
  In this paper, the notion of weak solutions is introduced
  in terms of a limit of classical (resp.~strong) solutions
  of the initial-value problem under consideration.

  For \emph{linear Ohm laws} $j_1 = \sigma(x,t) \, e$ (see Section~1 above),
  problem~(1.1)--(1.4) is included in the work~[13].
\end{remark}
\begin{remark}
  Suppose that hypotheses~(H1)--(H6) hold true. In addition, assume
  \begin{equation*}
    j(x,t,\xi) \cdot \xi
    \, \ge \, 0
    \quad\text{for all $(x,t,\xi) \in Q_T \times \R^3$}
  \end{equation*}
  (cf.~Examples~1 and~2 in Section~1).
  Then any weak solution $(e,h) \in L^2(Q_T)^3 \times L^2(Q_T)^3$
  of~(1.1)--(1.4) satisfies the \emph{energy inequality}
  \begin{equation}
    \frac{1}{2}
    \Big( \big\| \hat{e}(t) \big\|{}^2_{H_\varepsilon}
    + \big\| \hat{h}(t) \big\|{}^2_{H_\mu} \Big)
    \, \le \, \frac{1}{2}
    \Big( \| e_0 \|^2_{H_\varepsilon} + \| h_0 \|^2_{H_\mu} \Big)
    \quad\text{for all $t \in [0,T]$}
  \end{equation}
  (cf.~also~[12, Corollary~7.6, p.~329]).
  Thus, for current density fields $j = j_1 = \sigma(x,t) \, e$
  ($\sigma(x,t)$ being a symmetric non-negative $3 \times 3$ matrix
  with bounded measurable entries), the uniqueness of weak solutions
  of~(1.1)--(1.4) follows from~(4.4).
  We note that this uniqueness result is a special case
  of~Theorem~4.2 (well-posedness of~(1.1)--(1.4))
  provided the mapping $\xi \longmapsto j(\cdot,\cdot,e)$ is monotone
  (cf.~condition~(b) in Section~1).
\end{remark}
\begin{remark}
  Assume~(H2), (H3) and let $j = j_1 = \sigma(x) \, e$,
  where $\sigma(x) = (\sigma_{kl}(x))_{k,l = 1,2,3}$ $(x \in \Omega)$
  is \emph{any} matrix with bounded measurable entries.

  Let $(e,h) \in L^2(Q_T)^3 \times L^2(Q_T)^3$ be a weak solution
  of~(1.1)--(1.4) with initial data
  \begin{equation*}
    e_0 \, = \, h_0 \, = \, 0
    \quad\text{a.e.~in $\Omega$}.
  \end{equation*}
  Then
  \begin{equation*}
    e \, = \, h \, = \, 0
    \quad\text{a.e.~in $Q_T$}.
  \end{equation*}
  This result has been proved in~[25]
  by deriving an energy equality for the primitives
  $\int_0^t e(\cdot,s) \ds$, $\int_0^t h(\cdot,s) \ds$ $(t \in [0,T])$
  and then applying the Gronwall lemma
  (cf.~also~[12, pp.~330--331], [21, Ch.~3, \S~8.2]).

  An analogous uniqueness result has been presented
  in~[10, Ch.~VII, \S~4.3] the proof of which makes use
  of an approximation technique for weak solutions of~(1.1)--(1.4)
  that is similar to ours in Section~3.
\end{remark}
From Theorem~4.1 we deduce
\begin{theorem}[Well-posedness of~(1.1)--(1.4)]
  Assume~\emph{(H1)--(H3)} and~\emph{(H5), (H6)}.
  In addition, suppose that
  \begin{equation}
    \big( j(x,t,\xi) - j(x,t,\eta) \big) \cdot \big( \xi - \eta \big)
    \, \ge \, 0
    \quad\text{for all $(x,t) \in Q_T$ and all $\xi$, $\eta \in \R^3$}
    \tag{H7}
  \end{equation}
  \emph{(}cf.~condition~\emph{(b)} in Section~1\,\emph{)}.
  
  Let $\big(e^{(k)},h^{(k)}\big) \in L^2(Q_T)^3 \times L^2(Q_T)^3$
  $(k = 1,2)$ be weak solutions of \emph{(1.1)--(1.4)}
  that correspond to initial data
  $\big(e_0{}^{(k)},h_0{}^{(k)}\big) \in L^2(\Omega)^3 \times L^2(\Omega)^3$
  $(k = 1,2)$, respectively.

  Then, for all $t \in [0,T]$,
  \begin{equation}
    \big\| e^{(1)}(t) - e^{(2)}(t) \big\|{}^2_{H_\varepsilon}
    + \big\| h^{(1)}(t) - h^{(2)}(t) \big\|{}^2_{H_\mu}
    \, \le \, \big\| e_0{}^{(1)} - e_0{}^{(2)} \big\|{}^2_{H_\varepsilon}
    + \big\| h_0{}^{(1)} - h_0{}^{(2)} \big\|{}^2_{H_\mu}.
  \end{equation}
  \emph{(}On the left side of~\emph{(4.5)}
  the continuous representatives of $e^{(k)}, h^{(k)}$
  according to Theorem~3.1 are understood,
  where the symbol $\,\hat{}$ is omitted for notational simplicity.\emph{)}
\end{theorem}
\begin{proof}
  We consider integral identities~(2.9), (2.10)
  with $\big(e^{(1)},h^{(1)}\big)$ as well as $\big(e^{(2)},h^{(2)}\big)$
  in place of $(e,h)$, and form the differences of the integral identities
  so obtained. Writing
  \begin{equation*}
    e_0^* \, = \, e_0{}^{(1)} - e_0{}^{(2)},
    \quad h_0^* \, = \, h_0{}^{(1)} - h_0{}^{(2)}
  \end{equation*}
  and
  \begin{equation*}
    e^* \, = \, e^{(1)} - e^{(2)},
    \quad h^* \, = \, h^{(1)} - h^{(2)},
    \quad g^* \, = \, j\big(e^{(1)}\big) - j\big(e^{(2)}\big),
  \end{equation*}
  we obtain
  \begin{align}
    & \left\{
      \begin{aligned}
        & - \displaystyle \int_{Q_T}
        (\varepsilon e^*) \cdot \partial_t \Phi \dx \dt
        + \int_{Q_T} \big( {-}h^* \cdot \curl \Phi
        + g^* \cdot \Phi \big) \dx \dt
        \, = \, \int_\Omega
        (\varepsilon e_0^*)(x) \cdot \Phi(x,0) \dx
        \\
        & \quad\text{for all $\Phi \in L^2(0,T;V_0)$
          such that $\partial_t \Phi \in L^2(Q_T)^3$
          and $\Phi(\cdot,T) = 0$ a.e.~in $\Omega$},
      \end{aligned}
          \right.
    \\[\medskipamount]
    & \left\{
      \begin{aligned}
        & - \displaystyle \int_{Q_T}
        (\mu h^*) \cdot \partial_t \Psi \dx \dt
        + \int_{Q_T} e^* \cdot \curl \Psi \dx \dt
        \, = \, \int_\Omega
        (\mu h_0^*)(x) \cdot \Psi(x,0) \dx
        \\
        & \quad\text{for all $\Psi \in L^2(0,T;V)$
          such that $\partial_t \Psi \in L^2(Q_T)^3$
        and $\Psi(\cdot,T) = 0$ a.e.~in $\Omega$},
    \end{aligned}
          \right.
  \end{align}
  i.e. $(e^*,h^*) \in L^2(Q_T)^3 \times L^2(Q_T)^3$ is a weak solution
  of~(1.1)--(1.4) with $j = j_0 + j_1$,
  $j_0 = g^*$ and $j_1 = 0$ (cf.~(H2), (H3)).
  Hence, Theorem~4.1 applies to~(4.6), (4.7).
  Then, the energy equality~(4.1) reads
  \begin{equation*}
    \frac{1}{2}
    \Big( \| e^*(t) \|^2_{H_\varepsilon}
    + \| h^*(t) \|^2_{H_\mu} \Big)
    + \int_0^t (g^*,e^*)_H \ds
    \, = \, \frac{1}{2}
    \Big( \| e_0^* \|^2_{H_\varepsilon} + \| h_0^* \|^2_{H_\mu} \Big)
    \quad\text{for all $t \in [0,T]$}.
  \end{equation*}
  Observing~(H7) we obtain~(4.5). The proof is complete.
\end{proof}
\begin{remark}
  Theorem~4.2 represents a special case
  of the notion of well-posedness of evolution problems
  discussed in~[27, p.~404, p.~413].
\end{remark}
\section{Existence of weak solutions of~(1.1)--(1.4) \linebreak
  via the Faedo-Galerkin method}
\noindent
The Faedo-Galerkin method is widely used for solving evolution problems.
From the wealth of literature we only refer to
[21, Ch.~3, \S\S~8.1--8.2], [22, Ch.~2, \S~1.2] and~[33, Ch.~30, \S\S~1--3].

In~[10, Ch.~VII, \S\S~4.1--4.3] the authors used this method
for the proof of the existence of weak solutions of~(1.1)--(1.4)
with \emph{linear Ohm laws} $j = j_1 = \sigma_0(x) \, e$.
The following theorem extends this result to the class of
\emph{nonlinear Ohm laws} we have introduced by hypotheses
(H2), (H3) (cf.~Section~2).
\begin{theorem}
  Assume~\emph{(H1)--(H3)} and~\emph{(H5)--(H7)}.
  Then for every $(e_0,h_0) \in H \times H$ there exists a weak solution
  \begin{equation*}
    (e,h) \in L^\infty(0,T;H) \times L^\infty(0,T;H)
  \end{equation*}
  of~\emph{(1.1)--(1.4)} which satisfies the estimate
  \begin{equation}
    \|e(t)\|^2_{H_\varepsilon} + \|h(t)\|^2_{H_\mu}
    \le c \left( \|e_0\|^2_{H_\varepsilon}
    + \|h_0\|^2_{H_\mu} + \|j_0\|^2_{L^2(Q_T)^3} \right) 
    \quad\text{for a.e.~$t \in [0,T]$},
  \end{equation}
  where $c = \const > 0$ depends on $c_1$ and $\varepsilon_*$
  from~\emph{(H3)} and~\emph{(H7)}, respectively, and on~$T$.
\end{theorem}
\noindent
For what follows we introduce more notations.

The separability of $V_0$ and $V$ implies the existence of sequences
$(\varphi_k)_{k \in \N} \subset V_0$ and $(\psi_k)_{k \in \N} \subset V$
such that
\begin{equation*}
  \text{$\{ \varphi_1, \dots, \varphi_m \}$ and $\{ \psi_1, \dots, \psi_m \}$
  are linearly independent for every $m \in \N$};
\end{equation*}
\begin{equation}
  \overline{\bigcup_{m=1}^\infty X_m} \, = \, V_0,
  \quad \overline{\bigcup_{m=1}^\infty Y_m} \, = \, V,
\end{equation}
where
\begin{equation*}
  X_m \, := \, \hull \{ \varphi_1, \dots, \varphi_m \},
  \quad Y_m \, := \, \hull \{ \psi_1, \dots, \psi_m \}.
\end{equation*}
Without any loss of generality, we may assume that 
\begin{equation}
  (\varphi_k,\varphi_l)_{H_\varepsilon} \, = \, \delta_{kl},
  \quad (\psi_k,\psi_l)_{H_\mu} \, = \, \delta_{kl}
  \quad\text{for all $k, l \in \N$ ($\delta_{kl}$ Kronecker's delta)}.
\end{equation}
\subsubsection*{Proof of Theorem~5.1}
We proceed in five steps.
\subsubsection*{Step~1.~Defining Faedo-Galerkin approximations
  for~\emph{(1.1)--(1.4)}}
For $m \in \N$ we define approximations by
\begin{equation*}
  e_m(t) \, := \, \sum_{k=1}^m a_{m,k}(t) \, \varphi_k,
  \quad h_m(t) \, := \, \sum_{k=1}^m b_{m,k}(t) \, \psi_k,
  \quad t \in [0,T],
\end{equation*}
where the real-valued functions $a_{m,k} = a_{m,k}(t)$, $b_{m,k} = b_{m,k}(t)$
will be determined by the following system of ordinary differential equations
\begin{align}
  \dot{a}_{m,k}(t)
  \, & = \, \big(\!\curl h_m(t) - j(e_m(t)), \varphi_k \big){}_H
  \\[\smallskipamount]
  \dot{b}_{m,k}(t)
  \, & = \, - \big(\!\curl e_m(t), \psi_k \big){}_H  
\end{align}
($t \in [0,T]$, $k = 1, \dots, m$).

To formulate initial conditions for $(a_{m,k}, b_{m,k})$,
we combine~(5.2) and the density of $V_0$ and $V$ in $H$ to obtain
real numbers $(\alpha_{m,k}, \beta_{m,k})$ $(k = 1, \dots, m$) such that
\begin{equation*}
  \sum_{k=1}^m \alpha_{m,k} \, \varphi_k \longrightarrow e_0,
  \quad \sum_{k=1}^m \beta_{m,k} \, \psi_k \longrightarrow h_0
  \quad\text{in $H$ as $m \to \infty$}.
\end{equation*}
We now complement system~(5.4), (5.5) by the initial conditions
\begin{equation}
  a_{m,k}(0) \, = \, \alpha_{m,k},
  \quad b_{m,k}(0) \, = \, \beta_{m,k}
  \quad (k = 1, \dots, m).
\end{equation}
It follows
\begin{equation}
  e_m(0) \longrightarrow e_0,
  \quad h_m(0) \longrightarrow h_0
  \quad\text{in $H$ as $m \to \infty$}.
\end{equation}

We establish the existence of real-valued, absolutely continuous functions
\begin{equation*}
  \big( a_{m,1}, \dots, a_{m,m}, b_{m,1}, \dots, b_{m,m} \big)
\end{equation*}
on the interval $[0,T]$ that satisfies equation~(5.4), (5.5)
for a.e.~$t \in [0,T]$ and attain initial values~(5.6).

To this end, we introduce a mapping
\begin{equation*}
  f_m: [0,T] \times (\R^m \times \R^m) \longrightarrow \, \R^m \times \R^m
\end{equation*}
as follows. For $(t,\xi,\eta) \in [0,T] \times (\R^m \times \R^m)$ let
\begin{equation*}
  f_m(t,\xi,\eta) \, := \,
  \begin{pmatrix}
    \sum_{l=1}^m (\curl \psi_l, \varphi_1)_H \, \eta_l
    - \big( j \big( \cdot, t, \sum_{k=1}^m \xi_k \varphi_k \big),
        \varphi_1 \big){}_H
    \\[\smallskipamount]
    \vdots
    \\[\smallskipamount]
    \sum_{l=1}^m (\curl \psi_l, \varphi_m)_H \, \eta_l
    - \big( j \big( \cdot, t, \sum_{k=1}^m \xi_k \varphi_k \big),
        \varphi_m \big){}_H
    \\[\bigskipamount]
    - \sum_{l=1}^m (\curl \varphi_l, \psi_1)_H \, \xi_l
    \\[\smallskipamount]
    \vdots
    \\[\smallskipamount]
    - \sum_{l=1}^m (\curl \varphi_l, \psi_m)_H \, \xi_l
  \end{pmatrix}.
\end{equation*}
Defining
\begin{equation*}
  y_m \, := \, (a_m,b_m),
\end{equation*}
we may write~(5.4)--(5.6) in the form
\begin{gather}
  \dot{y}_m(t)
  \, = \, f_m(t,y_m(t))
  \quad\text{for $t \in [0,T]$},
  \\[\smallskipamount]
  y_m(0)
  \, = \, (\alpha_m, \beta_m)
\end{gather}
($\alpha_m$, $\beta_m$ as in~(5.6)).

The following properties of $f_m$ are readily seen:
\begin{itemize}
\item[(i)] $t \longmapsto f_m(t,\xi,\eta)$ is \emph{measurable}
  on $[0,T]$ for all $(\xi,\eta) \in \R^m \times \R^m$;
  \smallskip
\item[(ii)] $(\xi,\eta) \longmapsto f_m(t,\xi,\eta)$ is \emph{continuous}
  on $\R^m \times \R^m$ for a.e.~$t \in [0,T]$;
  \smallskip
\item[(iii)] there exists $k_m = \const > 0$ such that  
  \begin{equation*}
    |f_m(t,\xi,\eta)|
    \, \le \, k_m \big( \|j_0(\cdot,t)\|_H + |\xi| + |\eta| \big)
  \end{equation*}
  for a.e.~$t \in [0,T]$ and all $(\xi,\eta) \in \R^m \times \R^m$
\end{itemize}
(cf.~Appendix~A below).
Indeed, to verify (i), (ii) it is evidently sufficient
to note that the functions
\begin{equation*}
  \textstyle
  (t,\xi) \longmapsto
  \big( j \big( \cdot, t, \sum_{k=1}^m \xi_k \varphi_k \big),
  \varphi_l \big){}_H,
  \quad (t,\xi) \in [0,T] \times \R^m,
  \quad l = 1, \dots, m
\end{equation*}
satisfy (i), (ii). This can be easily derived from~(H2), (H3)
be the aid of the Fubini theorem and the Lebesgue dominated convergence theorem.
Appealing once more to these hypotheses
on obtains the bounds on $|f_m|$ in~(iii).

We are now in a position to apply an existence result for solutions
to the Cauchy problem for ordinary differential equations
(cf.~Appendix~A, Theorem~A.2).
From this result it follows that there exists
an absolutely continuous function $y_m:[0,T] \longrightarrow \R^m \times \R^m$
that satisfies system~(5.8) for a.e.~$t \in [0,T]$
and attains initial value~(5.9).
Thus, defining functions $(a_m,b_m)$ by
\begin{equation*}
  a_{m,k} \, := \, y_{m,k},
  \quad b_{m,l} \, := \, y_{m,m+l}
  \quad\text{for $k,l = 1, \dots, m$},
\end{equation*}
we obtain a solution of~(5.4)--(5.6).
\subsubsection*{Step~2.~A-priori estimates}
First, observing~(5.3) we may write~(5.4), (5.5) in the form
\begin{align}
  \big(\dot{e}_m(s), \varphi_k\big){}_{H_\varepsilon}
  + \big({-}\curl h_m(s) + j(e_m(s)), \varphi_k\big){}_H
  \, & = \, 0,
  \\[\smallskipamount]
  \big(\dot{h}_m(s), \psi_l\big){}_{H_\mu}
  + \big(\!\curl e_m(s), \psi_l \big){}_H  
  \, & = \, 0
\end{align}
for a.e.~$s \in [0,T]$ ($m \in \N$; $k,l = 1, \dots, m$).
We multiply~(5.10) by $a_{m,k}(s)$, (5.11) by $b_{m,l}(s)$,
sum for $k,l = 1, \dots, m$, make then use of the identity
\begin{equation*}
  \int_\Omega (\curl e_m(s)) \cdot h_m(s) \ds
  \, = \, \int_\Omega e_m(s) \cdot \curl h_m(s) \ds,
  \quad s \in [0,T],
\end{equation*}
integrate the equations obtained in this way
over the interval $[0,t]$ ($t \in [0,T]$) and integrate by parts
with respect to $s$ the terms involving $\dot{e}_m(s)$ and $\dot{h}_m(s)$.
To estimate the integral $\int_0^t (j(e_m),e_m)_H \ds$,
we use hypotheses~(H2), (H3) and~(H6). It follows 
\begin{align*}
  \|e_m(t)\|^2_{H_\varepsilon} + \|h_m(t)\|^2_{H_\mu}
  \, & = \, \|e_m(0)\|^2_{H_\varepsilon} + \|h_m(0)\|^2_{H_\mu}
       - 2 \int_0^t (j(e_m),e_m)_H \ds
  \\
     & \le \, \|e_m(0)\|^2_{H_\varepsilon} + \|h_m(0)\|^2_{H_\mu}
       + c_2 \Bigg( \|j_0\|^2_{L^2(Q_T)^3}
       + \int_0^t \|e_m\|^2_{H_\varepsilon} \ds \Bigg)
\end{align*}
for all $t \in [0,T]$
($c_2 = \const > 0$ depending only on the constants
$c_1$ and $\varepsilon_*$ from~(H3) and~(H6), respectively).
Thus, by the Gronwall lemma (cf.~Appendix~A below),
\begin{equation}
  \|e_m(t)\|^2_{H_\varepsilon} + \|h_m(t)\|^2_{H_\mu}
  \, \le \, c_3 \left( \|e_m(0)\|^2_{H_\varepsilon}
    + \|h_m(0)\|^2_{H_\mu} + \|j_0\|^2_{L^2(Q_T)^3} \right)
\end{equation}
for all $t \in [0,T]$ and all $m \in \N$
($c_3 = \const > 0$ depending on $c_2$ as well as on $T$).
\subsubsection*{Step~3.~Passing to the limits as $m \to \infty$}
In view of~(5.7) the right-hand side of~(5.12)
is uniformly bounded with respect to $m \in \N$.
Thus, from~(5.12) we conclude that there exists a subsequence
of $(e_m,h_m)$ (not relabelled) and elements
\begin{equation*}
  e, h \in L^\infty(0,T;H),
  \quad v, w \in H,
  \quad \chi \in L^2(Q_T)^3
\end{equation*}
such that
\begin{gather}
  e_m \longrightarrow e,
  \quad h_m \longrightarrow h
  \quad\text{weakly* in $L^\infty(0,T;H)$},
  \\[\smallskipamount]
  e_m(T) \longrightarrow v,
  \quad h_m(T) \longrightarrow w
  \quad\text{weakly in $H$},
  \\[\smallskipamount]
  j(e_m) \longrightarrow \chi
  \quad\text{weakly in $L^2(Q_T)^3$}
\end{gather}
as $m \to \infty$. Moreover, passing to the limits in~(5.12)
as $m \to \infty$ we find~(5.1) (with $c = c_3$).

Let $N \in \N$. Given $m > N$, in~(5.10), (5.11)
we only consider equations with indices $k = 1, \dots, N$.
By the definition of $X_N, Y_N$, for a.e.~$t \in [0,T]$,
\begin{alignat}{2}
  \big(\dot{e}_m(t), \varphi\big){}_{H_\varepsilon}
  + \big({-}\curl h_m(t) + j(e_m(t)), \varphi\big){}_H
  \, & = \, 0
  & \quad & \text{for any $\varphi \in X_N$},
  \\[\smallskipamount]
  \big(\dot{h}_m(t), \psi\big){}_{H_\mu}
  + \big(\!\curl e_m(t), \psi\big){}_H  
  \, & = \, 0
  & \quad & \text{for any $\psi \in Y_N$}.
\end{alignat}
Next, let $\zeta, \theta \in C^1([0,T])$.
We multiply~(5.16) by $\zeta(t)$, (5.17) by $\theta(t)$,
integrate over the interval $[0,T]$ and integrate by parts
with respect to $t$ the terms involving $\dot{e}_m(t)$ and $\dot{h}_m(t)$.
Using~(5.7) and~(5.13)--(5.15) we obtain upon letting tend $m \to \infty$
\begin{align}
  & \left\{
    \begin{aligned}
      & (v,\varphi)_{H_\varepsilon} \, \zeta(T)
      - (e_0,\varphi)_{H_\varepsilon} \, \zeta(0)
      - \int_0^T (e(t),\varphi)_{H_\varepsilon} \, \dot{\zeta}(t) \dt
      \\
      & + \int_0^T \big({-}(h(t),\curl \varphi)_H + (\chi(t),\varphi)_H \big)
      \, \zeta(t) \dt \, = \, 0
      \quad\text{for any $\varphi \in X_N$},
    \end{aligned}
        \right.
  \\[\medskipamount]
  & \left\{
    \begin{aligned}
      & (w,\psi)_{H_\mu} \, \theta(T)
      - (h_0,\psi)_{H_\mu} \, \zeta(0)
      - \int_0^T (h(t),\psi)_{H_\mu} \, \dot{\theta}(t) \dt
      \\
      & + \int_0^T (e(t),\curl \psi)_H \, \theta(t) \dt \, = \, 0
      \quad\text{for any $\psi \in Y_N$}.
    \end{aligned}
        \right.
\end{align}
From~(5.2) it follows that~(5.18), (5.19)
continue to hold true for any $\varphi \in V_0$ resp.~$\psi \in V$.

To proceed, let $\zeta, \theta \in C^1([0,T])$ be such that
$\zeta(T) = \theta(T) = 0$. Then~(5.18), (5.19) can be viewed
as a variant of~(2.9), (2.10) with $j_0 = \chi$, $j_1 = 0$
(take $\Phi(x,t) = \varphi(x) \, \zeta(t)$,
$\Psi(x,t) = \psi(x) \, \theta(t)$, $(x,t) \in Q_T$).
From Theorem~2.1 and its proof it follows
that there exist the distributional derivatives
\begin{equation*}
  (\varepsilon e)' \in L^2(0,T;V_0^*),
  \quad (\mu h)' \in L^2(0,T;V^*),
\end{equation*}
where, for a.e.~$t \in [0,T]$,
\begin{alignat}{2}
  \langle (\varepsilon e)'(t), \varphi \rangle_{V_0}
  - (h(t), \curl \varphi)_H
  + (\chi(t), \varphi)_H
  \, & = \, 0
  & \quad & \text{for all $\varphi \in V_0$},
  \\[\smallskipamount]
  \langle (\mu h)'(t), \psi \rangle_V
  + (e(t), \curl \psi)_H
  \, & = \, 0
  & \quad & \text{for all $\psi \in V$}.
\end{alignat}
In addition, the continuous representatives $\hat{e}, \hat{h} \in C([0,T];H)$
in the equivalence classes $e, h \in L^\infty(0,T;H)$ attain
the initial values $\hat{e}(0) = e_0$, $\hat{h}(0) = h_0$ in $H$
and satisfy the \emph{energy equality}
\begin{equation}
  \frac{1}{2}
  \Big( \| \hat{e}(t) \|^2_{H_\varepsilon}
  + \| \hat{h}(t) \|^2_{H_\mu} \Big)
  + \int_0^t (\chi,e)_H \ds
  \, = \, \frac{1}{2}
  \Big( \| e_0 \|^2_{H_\varepsilon}
    + \| h_0 \|^2_{H_\mu} \Big)
  \quad\text{for all $t \in [0,T]$}
\end{equation}
(see Theorem~3.1 and Theorem~4.1).
\subsubsection*{Step~4.~Proof of $v = \hat{e}(T)$, $w = \hat{h}(T)$}
We consider~(5.18), (5.19) with $\zeta, \theta \in C^1([0,T])$
satisfying $\zeta(0) = \theta(0) = 0$ and $\zeta(T) = \theta(T) = 1$.
It follows
\begin{align}
  & \left\{
    \begin{aligned}
      & (v,\varphi)_{H_\varepsilon}
      - \int_0^T (e(t),\varphi)_{H_\varepsilon} \, \dot{\zeta}(t) \dt
      + \int_0^T \big({-}(h(t),\curl \varphi)_H + (\chi(t),\varphi)_H \big)
      \, \zeta(t) \dt \, = \, 0
      \\
      & \; \text{for any $\varphi \in V_0$},
    \end{aligned}
        \right.
  \\[\medskipamount]
  & \left\{
    \begin{aligned}
      & (w,\psi)_{H_\mu}
      - \int_0^T (h(t),\psi)_{H_\mu} \, \dot{\theta}(t) \dt
      + \int_0^T (e(t),\curl \psi)_H \, \theta(t) \dt \, = \, 0
      \\
      & \; \text{for any $\psi \in V$}.
    \end{aligned}
        \right.
\end{align}
Thus, by~(5.23), for any $\varphi \in V_0$,
\begin{align*}
  (v,\varphi)_{H_\varepsilon}
  - \int_0^T (e(t),\varphi)_{H_\varepsilon} \, \dot{\zeta}(t) \dt
  \, & = \, \int_0^T \langle (\varepsilon e)'(t), \varphi \rangle_{V_0}
       \quad \text{(by~(5.20))}
  \\
  & = \, (\hat{e}(T),\varphi)_{H_\varepsilon}
  - \int_0^T (e(t),\varphi)_{H_\varepsilon} \, \dot{\zeta}(t) \dt
\end{align*}
(by integration by parts~(2.13), and~(2.11)).
Whence, $v = \hat{e}(T)$ in $H$. The claim $w = \hat{h}(T)$ in $H$
follows from~(5.24) and~(5.21) by an analogous argument.
\subsubsection*{Step~5.~Proof of $\chi = j(e)$}
To begin with, we note that
\begin{align*}
  \limsup_{m \to \infty} \int_0^T (j(e_m),e_m)_H \ds
  \, & \le \, \frac{1}{2}
       \Big( \|e_0\|^2_{H_\varepsilon} + \|h_0\|^2_{H_\mu} \Big)
  - \frac{1}{2} \liminf_{m \to \infty} 
  \Big( \|e_m(T)\|^2_{H_\varepsilon} + \|h_m(T)\|^2_{H_\mu} \Big)
  \\
  & \le \, \frac{1}{2} \Big( \|e_0\|^2_{H_\varepsilon} + \|h_0\|^2_{H_\mu} \Big)
    - \frac{1}{2}
    \Big( \|\hat{e}(T)\|^2_{H_\varepsilon} + \|\hat{h}(T)\|^2_{H_\mu} \Big)
\end{align*}
(by~(5.14) and $v = \hat{e}(T)$, $w = \hat{h}(T)$ (see Step~4)).
Hence, using energy equality~(5.22) for $t = T$, we get
\begin{equation}
  \limsup_{m \to \infty} \int_0^T (j(e_m),e_m)_H \ds
  \, \le \, \int_0^T (\chi,e)_H \ds
\end{equation}

Finally, let $z \in L^2(Q_T)^3$ and $\lambda > 0$. The monontonicity
of $\xi \longmapsto j(\cdot,\cdot,\xi)$ (cf.~(H7)) implies
\begin{equation*}
  \int_0^T \big( j(e_m) - j(e - \lambda z),
  e_m - (e - \lambda z )\big){}_H \ds 
  \, \ge \, 0.
\end{equation*}
Using~(5.13), (5.15) and~(5.25)
we find upon letting tend $m \to \infty$ and then dividing by $\lambda$
\begin{equation}
  \int_0^T (\chi - j(e - \lambda z), z)_H \ds 
  \, \ge \, 0.
\end{equation}
Now, hypotheses~(H2), (H3) allow us to make use
of the Lebesgue dominated convergence theorem for the passage to limit
as $\lambda \to 0$ in~(5.26). It follows
\begin{equation*}
  \chi = j(e)
  \quad\text{a.e.~in $Q_T$}.
\end{equation*}
The proof of Theorem~5.1 is complete.
\begin{remark}
  The uniqueness of weak solutions $(e,h)$ of~(1.1)--(1.4)
  (cf. Section~4) implies the convergence of the \emph{whole}
  sequence of Faedo-Galerkin approximations $(e_m,h_m)$ to $(e,h)$.
\end{remark}
\begin{remark}
  We note that the mapping $j:L^2(Q_T)^3 \longrightarrow L^2(Q_T)^3$
  is a special case of an \emph{operator of type} $(M)$. Our above reasoning
  for proving $\chi = j(e)$ is a variant of the well-known ``Minty trick''
  (see~[22, p.~173], [33, p.~474]).
\end{remark}
\appendix
\section{On the solvability of the Cauchy problem \linebreak
  for an ordinary differential equation}
\noindent
In this appendix, we prove the existence of a solution of the Cauchy problem
\begin{equation}
  \dot{y}(t) \, = \, f(t,y(t))
  \quad\text{for $t \in [0,T]$},
  \quad y(0) \, = \, y_0
\end{equation}
of \textsc{C. Carath\'{e}odory}~[6, \S\S~576--592]
($0 < T < +\infty$, $y_0 \in \R^n$).
For this, we impose on the function $f:[0,T] \times \R^n \longrightarrow \R^n$
the conditions
\begin{align}
  & \text{$t \longmapsto f(t,\xi)$
  is measurable on $[0,T]$ for all $\xi \in \R^n$},
    \tag{a}
  \\[\smallskipamount]
  & \text{$\xi \longmapsto f(t,\xi)$
    is continuous on $\R^n$ for a.e.~$t \in [0,T]$}.
    \tag{b}
\end{align}
From~(a), (b) it follows that for any measurable function
$y:[0,T] \longrightarrow \R^n$ the function
\begin{equation*}
  t \longmapsto f(t,y(t)),
  \quad t \in [0,T]
\end{equation*}
is measurable on $[0,T]$ (see~[6, p.~665], [18, p.~195]).
Functions that satisfy conditions~(a), (b)
are usually called \emph{Carath\'{e}odory functions}.
\begin{theorem}
  Let $f:[0,T] \times \R^n \to \R^n$ satisfy conditions
  \emph{(a), (b)} and suppose that there exists a non-negative integrable
  function $A$ defined on $[0,T]$ such that
  \begin{equation*}
    |f(t,\xi)| \, \le \, A(t)
    \quad\text{for all $(t,\xi) \in [0,T] \times \R^n$}.
  \end{equation*}
  Then, for every $y_0 \in \R^n$ there exists an absolutely continuous
  function $y:[0,T] \longrightarrow \R^n$ that fulfills the equation
  in~\emph{(A.1)} for a.e.~$t \in [0,T]$
  and attains the initial value $y(0) = y_0$.
\end{theorem}
For proofs see~[6, pp.~668--672, Satz~2] as well as~[18, pp.~193--197, Satz~1].
We note that these proofs yield in one step the existence
of a solution of~(A.1) on the \emph{whole} interval $[0,T]$.
In~[7, pp.~43--44, Thm.~1.1], the authors prove an existence result for~(A.1)
on some subinterval $[0,T_0]$ ($0 < T_0 \le T$).

We now present an extension of Theorem~A.1 for functions~$f$
with a more general growth with respect to~$(t,\xi)$.
This result implies straightforwardly the existence
of Faedo-Galerkin approximations we used in the proof of Theorem~5.1.
\begin{theorem}
  Let $f:[0,T] \times \R^n \longrightarrow \R^n$ satisfy conditions
  \emph{(a), (b)} and suppose that there are a non-negative integrable
  function $A$ defined on $[0,T]$, and $C_0 = \const > 0$ such that
  \begin{equation}
    |f(t,\xi)| \, \le \, A(t) + C_0 |\xi|
    \quad\text{for all $(t,\xi) \in [0,T] \times \R^n$}.
    \tag{c}
  \end{equation}
  Then the conclusion of Theorem~A.1 holds true.
\end{theorem}
For proving this result we will make use of the following
\begin{lemma*}[Gronwall]
  Let $c_1, c_2$ be non-negative constants.
  Let $u$ be a non-negative integrable function on $[0,T]$ such that
  \begin{equation*}
    u(t) \, \le \, c_1 + c_2 \int_0^t u(s) \ds
    \quad\text{for all $t \in [0,T]$}.
  \end{equation*}
  Then,
  \begin{equation*}
    u(t) \, \le \, c_1 \big( 1 + c_2 t \exp(c_2 t) \big)
    \quad\text{for all $t \in [0,T]$}.
  \end{equation*}
\end{lemma*}
\begin{proof}[Proof of Theorem~A.2]
  Fix any real number
  \begin{equation*}
    r > \big(1 + C_0 T \exp(C_0 T) \big)
    \int_0^T (A(t) + C_0 |y_0|) \dt
  \end{equation*}
  and, for any $(t,\xi) \in [0,T] \times \R^n$, define
  \begin{equation*}
    f_r(t,\xi) \, := \,
    \begin{cases}
      f(t,\xi)
      & \text{if $|\xi - y_0| \le r$}, \\
      f\!\left(t,y_0 + r \, \displaystyle\frac{\xi - y_0}{|\xi - y_0|}\right)
      & \text{if $|\xi - y_0| > r$}
    \end{cases}
  \end{equation*}
  (cf.~[18, p.~198]).
  The function $f_r$ satisfies conditions~(a), (b).
  From~(c) it follows
  \begin{align}
    |f_r(t,\xi)|
    \, & \le \, A(t) + C_0 (|y_0| + r), 
    \\[\smallskipamount]
    |f_r(t,\xi)|
    \, & \le \, A(t) + C_0 (|y_0| + |\xi - y_0|)
  \end{align}
  for all $(t,\xi) \in [0,T] \times \R^n$.

  Observing~(A.2), from Theorem~A.1
  we infer the existence of an absolutely continuous function
  $y_r:[0,T] \longrightarrow \R^n$ such that
  \begin{equation}
    y_r(t) \, = \, y_0 + \int_0^t f_r(s,y_r(s)) \ds
    \quad\text{for all $t \in [0,T]$}.
  \end{equation}
  By~(A.3),
  \begin{equation*}
    |y_r(t) - y_0|
    \, \le \, \int_0^t (A(s) + C_0 |y_0|) \ds
    + C_0 \int_0^t |y_r(s) - y_0| \ds
    \quad\text{for all $t \in [0,T]$}.
  \end{equation*}
  Thus, by the Gronwall lemma,
  \begin{equation*}
    |y_r(t) - y_0|
    \, \le \, \int_0^T (A(s) + C_0 |y_0|) \ds
    \, \big(1 + C_0 T \exp(C_0 T) \big)
    \, \le \, r
  \end{equation*}
  and therefore
  \begin{equation*}
    f_r(t,y_r(t)) \, = \, f(t,y_r(t))
    \quad\text{for all $t \in [0,T]$}.
  \end{equation*}
  Hence, the function $y := y_r$ satisfies~(A.1) for a.e.~$t \in [0,T]$,
  and $y(0) = y_0$. The proof of the theorem is complete.
\end{proof}
Finally, under significantly more general growth conditions on~$f$
than~(c) above, the existence of a solution of~(A.1)
on a subinterval $[0,T^*]$ ($0 < T^* < T$) has been proved
in~[6, pp.~681--682, Satz~6] and~[18, pp.~197--199, Satz~2].
For continuous functions $f:[0,T] \times \R^n \longrightarrow \R^n$
which satisfy slightly more general growth conditions than~(c) above,
the proof of the existence of a solution of~(A.1)
on the \emph{whole} interval $[0,T]$ has been formulated as Problem~5
in~[7, p.~61].
\end{document}